\documentclass[twoside,10pt]{siamltex}

\usepackage{amsfonts,amssymb}
\usepackage{amsmath}
\usepackage{graphicx}
\usepackage{color}
\usepackage{url}  
 
\newtheorem{thm}{Theorem}[section]

\newtheorem{prop}[thm]{Proposition}

\newtheorem{rem}[thm]{Remark}
\newtheorem{exm}[thm]{Example}

\newtheorem{algo}[thm]{Algorithm}
\numberwithin{equation}{section}

\newcommand{\abs}[1]{\left\vert#1\right\vert}

\newcommand{\mean}[1]{\mathbb{E}\lbrack #1\rbrack}

\makeatletter
\newcommand{\figcaption}{\def\@captype{figure}\caption}
\newcommand{\tabcaption}{\def\@captype{table}\caption}
\makeatother
 
\def\@captype{table}
\def\@captype{table}
\def\cprime{$'$}
 
 \allowdisplaybreaks[1]
 
\begin{document}

\title{A Recursive sparse grid collocation method for differential
equations with white noise \footnotemark[1]}
\author{Z. Zhang\footnotemark[2]\ \footnotemark[3], M.V. Tretyakov\footnotemark[4], B.
Rozovskii\footnotemark[2]\ \footnotemark[5], \ and G.E. Karniadakis%
\footnotemark[2]\ \footnotemark[6]}
\maketitle

\begin{abstract}
We consider  
a sparse grid collocation method in conjunction with a time discretization of the
differential equations for computing expectations of functionals of solutions to differential
equations perturbed by time-dependent white noise. We first analyze the error of Smolyak's sparse grid
collocation used to evaluate expectations of functionals of solutions to
stochastic differential equations discretized by the Euler scheme. We show
theoretically and numerically that this algorithm can have satisfactory
accuracy for small magnitude of noise or  small integration time, however it
does not converge neither with decrease of the Euler scheme's time step size
nor with increase of Smolyak's sparse grid level. Subsequently,  we use this method as
a building block for proposing a new algorithm by combining sparse grid
collocation with a recursive procedure. This approach allows us
to numerically integrate linear stochastic partial differential equations
over longer times, which is illustrated in numerical tests on a stochastic
advection-diffusion equation.
\end{abstract}

\date{\today }

\renewcommand{\thefootnote}{\arabic{footnote}} \renewcommand{\thefootnote}{%
\fnsymbol{footnote}}

\footnotetext[1]{To cite this paper, use \\ Z. Zhang, M. V. Tretyakov, B. Rozovskii, and G. E. Karniadakis. A recursive sparse grid collocation method for differential equations with white noise. SIAM J. Sci. Comput., 36(4):A1652-A1677, 2014}
\footnotetext[2]{%
Division of Applied Mathematics, Brown University, Providence RI, 02912, USA}
 \footnotetext[3]{%
Email: zhongqiang\_zhang@brown.edu.}
\footnotetext[4]{%
School of Mathematical Sciences, University of Nottingham, Nottingham, NG7
2RD, UK. Email: Michael.Tretyakov@nottingham.ac.uk}
 \footnotetext[5]{%
Email: boris\_rozovsky@brown.edu.}
 \footnotetext[6]{%
Email: george\_karniadakis@brown.edu}



\section{Introduction}

In a number of applications from physics, financial engineering, biology and
chemistry it is of interest to compute expectations of some functionals of
solutions of ordinary stochastic differential equations (SDE) and stochastic
partial differential equations (SPDE) driven by white noise. Usually,
evaluation of such expectations requires to approximate solutions of
stochastic equations and then to compute the corresponding averages with
respect to the approximate trajectories. We will not consider the former in
this paper (see, e.g. \cite{MilTre-B04} and references therein) and will
concentrate on the latter. The most commonly used approach for computing the
averages is the Monte Carlo technique, which is known for its slow rate of
convergence and hence  limiting
computational efficiency of stochastic simulations. To speed up computation
of the averages, variance reduction techniques (see, e.g. \cite%
{MilTre-B04,MilTre09} and the references therein), quasi-Monte Carlo
algorithms \cite{Nie-B92,SloJoe-B94} {\color{black}and multi-level quasi-Monte Carlo methods \cite{KuoSS12a}}, and the multi-level Monte Carlo method 
\cite{Gil08,Gil13} have been proposed and used.

An alternative approach to computing the averages is (stochastic)
collocation methods in random space, which are deterministic methods in
comparison with the Monte Carlo-type methods that are based on a statistical
estimator of a mean. The expectation can
be viewed as an integral with respect to the measure corresponding to
approximate trajectories. In stochastic collocation methods, one uses
(deterministic) high-dimensional quadratures to evaluate these integrals. In
the context of uncertainty quantification where moments of stochastic
solutions are sought, collocation methods and their close counterparts
(e.g., Wiener chaos expansion-based methods) have been very effective in
reducing the overall computational cost in engineering problems, see e.g. 
\cite{GhaSpa-B91,TatMcR94,XiuKar02a}.

Stochastic equations or differential equations with randomness can be split
into differential equations perturbed by time-independent noise and by
time-dependent noise. It has been demonstrated in a number of works (see
e.g. \cite{BabTZ04,BieSch09,BabNT07,XiuHes05,MotNT12,NobTem09,ZhangGun12}
and references therein) that stochastic collocation methods can be a
competitive alternative to the Monte Carlo technique and its variants in the
case of differential equations perturbed by time-independent noise. The
success of these methods relies on smoothness in the random space and can
usually be achieved when it is sufficient to consider only a limited number
of random variables (i.e., in the case of a low dimensional random space).
The small number of random variables significantly limits the applicability
of stochastic collocation methods to differential equations perturbed by
time-dependent noise as, in particular, it will be demonstrated in this
paper.

The class of stochastic collocation methods for SDE with time-dependent
white noise includes cubatures on Wiener space \cite{LyoVic04},
derandomization \cite{MulRY12}, optimal quantization \cite{PagPha05,PagPri03}
and sparse grids of Smolyak type \cite{Ger-Phd07,GerGri98,GriHol10}. While
derandomization and optimal quantization aim at finding quadrature rules
which are in some sense optimal for computing a particular expectation under
consideration, cubatures on Wiener space and a stochastic collocation method
using Smolyak sparse grid quadratures (a sparse grid collocation method,
SGC) use pre-determined quadrature rules in a universal way without being
tailed towards a specific expectation {\color{black}unless some adaptive strategies are applied}. Since SGC is endowed with negative
weights, it is, in practice, different from cubatures on Wiener space, where
only quadrature rules with positive weights are used. Among quadrature
rules, SGC is of particular interest due to its computational convenience.
It has been considered in computational finance \cite{Ger-Phd07,GriHol10},
where  high accuracy was observed. We note that the use of SGC in \cite%
{Ger-Phd07,GriHol10} relies on exact sampling of geometric Brownian motion
and of solutions of other simple SDE models, i.e., SGC in these works was
not studied in conjunction with SDE approximations.

In this paper, we consider a SGC method accompanied by time discretization
of differential equations perturbed by time-dependent noise. Our objective
is twofold. {\em First}, using both analytical and numerical results, we warn that
straightforward carrying over stochastic collocation methods and, in
particular, SGC to the case of differential equations perturbed by
time-dependent noise (SDE or SPDE) usually leads to a failure. The main
reason for this failure is that when integration time increases and/or time
discretization step decreases, the number of random variables in
approximation of SDE and SPDE grows quickly. The number of collocation
points required for sufficient accuracy\ of collocation methods grows
exponentially with the number of random variables. This results in
{\color{black}failure} of algorithms based on SGC and SDE time discretizations.
Further, due to empirical evidence (see e.g. \cite{Pet03}), the use of SGC
is limited to problems with random space dimensionality of up to $40.$
Consequently, SGC algorithms for differential equations perturbed by
time-dependent noise can be used only over small time intervals unless a
cure for its fundamental limitation is found.

In Section~2 (after brief introduction to the sparse grid of Smolyak \cite%
{Smolyak63} (see also \cite{WasWoz95,GerGri98,XiuHes05}) and to the
weak-sense numerical integration for SDE (see, e.g. \cite{MilTre-B04})), we
obtain an error estimate for a SGC method accompanied by the Euler scheme
for evaluating expectations of smooth functionals of solutions of a scalar
linear SDE with additive noise. In particular, we conclude that the SGC can
successfully work for a small magnitude of noise and relatively short
integration time while {\color{black} in general} it does not converge neither with decrease of the
time discretization step used for SDE approximation nor with increase of the
level of Smolyak's sparse grid{\color{black}; see Remark \ref{rem:convergence-sgc}}. Numerical tests in Section~\ref%
{sec:num-experiments-sg} confirm our theoretical conclusions and we also
observe first-order convergence in time step size of the algorithm
using the SGC method as long as the SGC error is small relative to the error
of time discretization of SDE. We note that our conclusion is, to some
extent, similar to that for cubatures on Wiener space \cite{CasLit11}, for
Wiener chaos method \cite{HouLRZ06,LotMR97,LotRoz06,ZhangRTK12} and some
other functional expansion approaches \cite{BudKal96,BudKal97}.

The {\em second objective} of the paper is to suggest a possible cure for the
aforementioned deficiencies, which prevent SGC to be used over longer time
intervals. For longer time simulation, deterministic replacements (such as
stochastic collocation methods and functional expansion methods) of the
Monte Carlo technique in simulation of differential equations perturbed by
time-dependent noise do not work effectively unless some restarting
strategies allowing to `forget' random variables from earlier time steps are
employed. Examples of such strategies are the recursive approach for Wiener
chaos expansion methods to compute moments of solutions to{\ linear} SPDE 
\cite{LotMR97,ZhangRTK12} and an approach for cubatures on Wiener space
based on compressing the history data via a regression at each time step 
\cite{LitLyo12}.

Here we exploit the idea of the recursive approach to achieve accurate
longer time integration by numerical algorithms using the SGC. For linear
SPDE with {\em time-independent} coefficients, the recursive approach works as
follows. We first find an approximate solution of an SPDE at a relatively
small time $t=h$, and subsequently take the approximation at $t=h$ as the
initial value in order to compute the approximate solution at $t=2h$, and so
on, until we reach the final integration time $T=Nh$. To find second moments
of the SPDE solution, we store a covariance matrix of the approximate
solution at each time step $kh$ and recursively compute the first two
moments. Such an algorithm is proposed in Section~\ref%
{sec:recursive-sg-advdiff};   in Section~\ref{sec:num-experiments-sg} we
demonstrate numerically that this algorithm converges in time step $h$ and
that it can work well on longer time intervals. At the same time, a major
challenge remains: how to effectively use restarting strategies for SGC in
the case of nonlinear SDE and SPDE and further work is needed in this
direction.

\section{Sparse grid for weak integration of SDE}

\subsection{Smolyak's sparse grid\label{sec:ssg}}

Sparse grid quadrature is a certain reduction of product quadrature rules
which decreases the number of quadrature nodes and allows effective
integration in moderately high dimensions \cite{Smolyak63} (see also \cite%
{WasWoz95,NovRit99,GerGri98}). Here we introduce it in the form suitable for
our purposes.

We will be interested in evaluating $d$-dimensional integrals of a function $%
\varphi (y),$ $y\in \mathbb{R}^{d},$ with respect to a Gaussian measure: 
\begin{equation}
I_{d}\varphi :=\frac{1}{\left( 2\pi \right) ^{d/2}}\int_{\mathbb{R}%
^{d}}\varphi (y)\exp \left( -\frac{1}{2}\sum_{i=1}^{d}y_{i}^{2}\right)
\,dy_{1}\cdots dy_{d}.  \label{Defi}
\end{equation}%
Consider a sequence of one-dimensional Gauss--Hermite quadrature rules $%
Q_{n} $ with number of nodes $n\in \mathbb{N}$ for univariate functions $%
\psi (\mathsf{y}),$ $\mathsf{y\in }\mathbb{R}$: 
\begin{equation}
Q_{n}\psi (\mathsf{y})=\sum_{\alpha =1}^{n}\psi (\mathsf{y}_{n,\alpha })%
\mathsf{w}_{n,\alpha },  \label{eq:1d-Gauss-Hermite-rule}
\end{equation}%
where $\mathsf{y}_{n,1}<\mathsf{y}_{n,2}<\cdots <\mathsf{y}_{n,n}$ are the
roots of the Hermite polynomial 
$H_{n}(\mathsf{y})=(-1)^{n}e^{\mathsf{y}^{2}/2}\frac{d^{n}}{d\mathsf{y}^{n}}%
e^{-\mathsf{y}^{2}/2}$ 
and $\mathsf{w}_{n,\alpha }=n!/(n^{2}[H_{n-1}(\mathsf{y}_{n,\alpha })]^{2})$
are the associated weights. It is known that $Q_{n}\psi $ is exactly equal
to the integral $I_{1}\psi $ when $\psi $ is a polynomial of degree less
than or equal to $2n-1,$ i.e., the polynomial degree of exactness of
Gauss--Hermite quadrature rules $Q_{n}$ is equal to $2n-1.$

We can approximate the multidimensional integral $I_{d}\varphi $ by a
quadrature expressed as the tensor product rule 
\begin{eqnarray}
I_{d}\varphi &\approx &\bar{I}_{d}\varphi :=Q_{n}\otimes Q_{n}\cdots \otimes
Q_{n}\varphi (y_{1},y_{2},\cdots ,y_{d})=Q_{n}^{\otimes d}\varphi
(y_{1},y_{2},\cdots ,y_{d})  \label{appi} \\
&=&\sum_{\alpha _{1}=1}^{n}\cdots \sum_{\alpha _{d}=1}^{n}\varphi (\mathsf{y}%
_{n,\alpha _{1}},\ldots ,\mathsf{y}_{n,\alpha _{d}})\mathsf{w}_{n,\alpha
_{1}}\cdots \mathsf{w}_{n,\alpha _{d}},  \notag
\end{eqnarray}%
where for simplicity we use the same amount on nodes in all the directions.
The quadrature $\bar{I}_{d}\varphi $ is exact for all polynomials from the
space $\mathcal{P}_{k_{1}}\otimes \cdots \otimes \mathcal{P}_{k_{d}}$ with $%
\max k_{i}=2n-1,$ where $\mathcal{P}_{k}$ is the space of one-dimensional
polynomials of degree less than or equal to $k$ (we note in passing that
this fact is easy to prove using probabilistic representations of $%
I_{d}\varphi $ and $\bar{I}_{d}\varphi ).$ Computational costs of quadrature
rules are measured in terms of a number of function evaluations which is
equal to $n^{d}$ in the case of the tensor product (\ref{appi}), i.e., the
computational cost of (\ref{appi}) grows exponentially fast with dimension.

The sparse grid of Smolyak \cite{Smolyak63} reduces computational complexity
of the tensor product rule (\ref{appi}) via exploiting the difference
quadrature formulas: 
\begin{equation*}
A(L,d)\varphi :=\sum_{d\leq \left\vert \mathbf{i}\right\vert \leq
L+d-1}(Q_{i_{1}}-Q_{i_{1}-1})\otimes \cdots \otimes
(Q_{i_{d}}-Q_{i_{d}-1})\varphi ,
\end{equation*}%
where $Q_{0}=0$ and $\mathbf{i}=(i_{1},i_{2},\ldots ,i_{d})$ is a
multi-index with $i_{k}\geq 1$ and $\left\vert \mathbf{i}\right\vert
=i_{1}+i_{2}+\cdots +i_{d}$. The number $L$ is usually referred to as{\ the
level of the sparse grid}. The sparse grid rule (\ref{eq:smolyak-tensor-like}%
) can also be written in the following form \cite{WasWoz95}: 
\begin{equation}
A(L,d)\varphi =\sum_{L\leq \left\vert \mathbf{i}\right\vert \leq
L+d-1}(-1)^{L+d-1-\left\vert \mathbf{i}\right\vert }\binom{d-1}{\left\vert 
\mathbf{i}\right\vert -L}Q_{i_{1}}\otimes \cdots \otimes Q_{i_{d}}\varphi .
\label{eq:smolyak-tensor-like}
\end{equation}%
The quadrature $A(L,d)\varphi $ is exact for polynomials from the space $%
\mathcal{P}_{k_{1}}\otimes \cdots \otimes \mathcal{P}_{k_{d}}$ with $|%
\mathbf{k}|=2L-1,$ i.e., for polynomials of total degree up to $2L-1$ \cite[%
Corollary 1]{NovRit99}. Due to (\ref{eq:smolyak-tensor-like}), the total
number of nodes used by this sparse grid rule is estimated by 
\begin{equation*}
\#S\leq \sum_{L\leq \left\vert \mathbf{i}\right\vert \leq L+d-1}i_{1}\times
\cdots \times i_{d}.
\end{equation*}%
Table \ref{tbl:no-sgc-level5} lists the number of sparse grid points, $\#S$,
up to level 5 when the level is not greater than $d$.

\begin{table}[tbph]
\caption{The number of sparse grid points for the sparse grid quadrature 
\eqref{eq:smolyak-tensor-like} using the one-dimensional Gauss-Hermite
quadrature rule \eqref{eq:1d-Gauss-Hermite-rule}, when the sparse grid level 
$L$ $\leq d$.}
\label{tbl:no-sgc-level5}
\begin{center}
\scalebox{0.85}{
\begin{tabular}{c|ccccccc}
\hline
& $L=1$ & $L=2$ & $L=3$ & $L=4$ & $L=5$ &  &  \\ \hline
$\#S$ & $1$ & $2d+1$ & $2d^2+2d+1$ & $\frac{4}{3}d^3+2d^2+\frac{14}{3}d+1$ & 
$\frac{2}{3}d^4+\frac{4}{3}d^3 +\frac{22}{3}d^2+\frac{8}{3}d+1$ &  &  \\ 
\hline
\end{tabular}}
\end{center}
\end{table}

The quadrature $\bar{I}_{d}\varphi $ from (\ref{appi}) is exact for
polynomials of total degree $2L-1$ when $n=L.$ It is not difficult to see
that if the required polynomial exactness (in terms of total degree of
polynomials) is relatively small then the sparse grid rule (\ref%
{eq:smolyak-tensor-like}) substantially reduces the number of function
evaluations compared with the tensor-product rule (\ref{appi}). For
instance, suppose that the dimension $d=40$ and the required polynomial
exactness is equal to $3.$ Then the cost of the tensor product rule (\ref%
{appi}) is $3^{40} \doteq 1.\,215\,8\times 10^{19}$ while the cost of the
sparse grid rule (\ref{eq:smolyak-tensor-like}) based on one-dimensional
rule \eqref{eq:1d-Gauss-Hermite-rule} is $3281.$

\begin{rem}
In this paper we consider the isotropic SGC. More efficient algorithms might
be built using anisotropic SGC methods \cite{GriHol10,NobTW08}, which employ
more quadrature points along the \textquotedblleft most
important\textquotedblright\ direction. Goal-oriented quadrature rules, e.g. 
\cite{MulRY12,PagPha05,PagPri03}, can also be exploited instead of
pre-determined quadrature rules used here.   {\color{black} 
However, the effectiveness of   adaptive sparse grids relies heavily on the order of  importance in random dimension of    numerical solutions to stochastic differential equations, which is not always easy to reach. Furthermore, all these sparse grids as  integration methods in random space  grow quickly with random dimensions and thus cannot be used for longer time integration (usually with large random dimensions). Hence, we  consider only isotropic SGC.}
\end{rem}

\subsection{Weak-sense integration of SDE}

Let $(\Omega ,\mathcal{F},P)$ be a probability space and $(w(t),\mathcal{F}%
_{t}^{w})=((w_{1}(t),\ldots ,w_{r}(t))^{\intercal },\mathcal{F}_{t}^{w})$ be
an $r$-dimensional standard Wiener process, where $\mathcal{F}_{t}^{w},\
0\leq t\leq T,$ is an increasing family of $\sigma $-subalgebras of $%
\mathcal{F}$ induced by $w(t).$

Consider the system of Ito SDE 
\begin{equation}
dX=a(t,X)dt+\sum_{l=1}^{r}\sigma _{l}(t,X)dw_{l}(t),\ \ t\in (t_{0},T],\
X(t_{0})=x_{0},  \label{eq:ito-sde-vector}
\end{equation}%
where $X,$\ $a,$\ $\sigma _{r}$ are $m$-dimensional column-vectors and $%
x_{0} $ is independent of $w$. We assume that $a(t,x)$ and $\sigma (t,x)$
are sufficiently smooth and globally Lipschitz. We are interested in
computing the expectation 
\begin{equation}
u(x_{0})=\mathbb{E}f(X_{t_{0},x_{0}}(T)),  \label{eq:weak-apprx-def}
\end{equation}%
where $f(x)$ is a sufficiently smooth function with growth at infinity not
faster than a polynomial: 
\begin{equation}
\left\vert f(x)\right\vert \leq K(1+|x|^{\varkappa })  \label{f_cond}
\end{equation}%
for some $K>0$ and $\varkappa \geq 1.$

To find $u(x_{0}),$ we first discretize the solution of (\ref%
{eq:ito-sde-vector}). Let 
\begin{equation*}
h=(T-t_{0})/N,\ \ t_{k}=t_{0}+kh,\ \ k=0,\ldots ,N.
\end{equation*}%
In application to \eqref{eq:ito-sde-vector} the Euler scheme has the form 
\begin{equation}
X_{k+1}=X_{k}+a(t_{k},X_{k})h+\sum_{l=1}^{r}\sigma _{l}(t_{k},X_{k})\Delta
_{k}w_{l},  \label{eq:ito-sde-vector-euler}
\end{equation}%
where $X_{0}=x_{0}$ and $\Delta _{k}w_{l}=w_{l}(t_{k+1})-w_{l}(t_{k})$. The
Euler scheme can be realized in practice by replacing the increments $\Delta
_{k}w_{l}$ with Gaussian random variables: 
\begin{equation}
X_{k+1}=X_{k}+a(t_{k},X_{k})h+\sum_{l=1}^{r}\sigma _{l}(t_{k},X_{k})\sqrt{h}%
\xi _{l,k+1},  \label{eq:ito-sde-vec-strong-euler}
\end{equation}%
where $\xi _{r,k+1}$ are i.i.d. $\mathcal{N}(0,1)$ random variables. Due to
our assumptions, the following error estimate holds for %
\eqref{eq:ito-sde-vec-strong-euler} (see e.g. \cite[Chapter 2]{MilTre-B04}): 
\begin{equation}
|\mathbb{E}f(X_{N})-\mathbb{E}f(X(T))|\leq Kh,
\label{eq:ito-sde-vec-strongEuler-error}
\end{equation}%
where $K>0$ is a constant independent of $h$. This first-order weak
convergence can also be achieved by replacing $\xi _{l,k+1}$ with discrete
random variables \cite{MilTre-B04}, e.g., the weak Euler scheme has the form 
\begin{equation}
\tilde{X}_{k+1}=\tilde{X}_{k}+ha(t_{k},\tilde{X}_{k})+\sqrt{h}%
\sum_{l=1}^{r}\sigma _{l}(t_{k},\tilde{X}_{k})\zeta _{l,k+1},\ \ \ \
k=0,\ldots ,N-1,  \label{eq:ito-sde-vec-weak-euler}
\end{equation}%
where $\tilde{X}_{0}=x_{0}$ and $\zeta _{l,k+1}$ are i.i.d. random variables
with the law 
\begin{equation}
P(\zeta =\pm 1)=1/2.  \label{eq:weak-euler-discrete-apprx-gauss}
\end{equation}%
The following error estimate holds for \eqref{eq:ito-sde-vec-weak-euler}-%
\eqref{eq:weak-euler-discrete-apprx-gauss} (see e.g. \cite[Chapter 2]%
{MilTre-B04}):%
\begin{equation}
|\mathbb{E}f(\tilde{X}_{N})-\mathbb{E}f(X(T))|\leq Kh\ ,
\label{eq:ito-sde-vec-weakEuler-error}
\end{equation}%
where $K>0$ can be a different constant than in (\ref%
{eq:ito-sde-vec-strongEuler-error}).

Introducing the function $\varphi (y),$ $y\in $ $\mathbb{R}^{rN},$ so that 
\begin{equation}
\varphi (\xi _{1,1},\ldots ,\xi _{r,1},\ldots ,\xi _{1,N},\ldots ,\xi
_{r,N})=f(X_{N}),  \label{eq:1d-sde-functional-weak}
\end{equation}%
we have 
\begin{eqnarray}
u(x_{0}) &\approx &\bar{u}(x_{0}):=\mathbb{E}f(X_{N})=\mathbb{E}\varphi (\xi
_{1,1},\ldots ,\xi _{r,1},\ldots ,\xi _{1,N},\ldots ,\xi _{r,N})
\label{eq:strong-euler-mean-def} \\
&=&\frac{1}{(2\pi )^{rN/2}}\int_{\mathbb{R}^{rN}}\varphi (y_{1,1},\ldots
,y_{r,1},\ldots ,y_{1,N},\ldots ,y_{r,N})\exp \left( {-\frac{1}{2}%
\sum_{i=1}^{rN}y_{i}^{2}}\right) \,dy.  \notag
\end{eqnarray}%
Further, it is not difficult to see from \eqref{eq:ito-sde-vec-weak-euler}-%
\eqref{eq:weak-euler-discrete-apprx-gauss} and (\ref{appi}) that 
\begin{eqnarray}
u(x_{0}) &\approx &\tilde{u}(x_{0}):=\mathbb{E}f(\tilde{X}_{N})=\mathbb{E}%
\varphi (\zeta _{1,1},\ldots ,\zeta _{r,1},\ldots ,\zeta _{1,N},\ldots
,\zeta _{r,N})  \label{eq:ito-sde-weak-euler-functional-mean-tensor} \\
&=&Q_{2}^{\otimes rN}\varphi (y_{1,1},\ldots ,y_{r,1},\ldots ,y_{1,N},\ldots
,y_{r,N}),  \notag
\end{eqnarray}%
where $Q_{2}$ is the Gauss-Hermite quadrature rule {\color{black}defined in  \eqref{eq:1d-Gauss-Hermite-rule} with $n=2$
i.e., $\mathsf{y}_{1}=1$, $\mathsf{y}_2=-1$ with weights $\mathsf{w}_1=\mathsf{w}_2=1/2$).}
{\color{black}
Comparing  \eqref{eq:strong-euler-mean-def}  and   \eqref{eq:ito-sde-weak-euler-functional-mean-tensor}, we can say  that $\tilde{u}(x_{0})$  is a tensor-product quadrature rule for the multidimensional integral $\bar{u}(x_{0})$.  In other words, the weak Euler scheme  \eqref{eq:ito-sde-vec-weak-euler}-\eqref{eq:weak-euler-discrete-apprx-gauss}  
can be interpreted as  the strong Euler scheme with tensor-product integration in random space.We note that the approximation, $\tilde{u}(x_0)$, of  $\bar{u}(x_{0})$ satisfies (cf.  \eqref{eq:ito-sde-vec-strongEuler-error}  and  \eqref{eq:ito-sde-vec-weakEuler-error}) \begin{equation}\label{eq:weak-strong-order-one}
\bar{u}(x_{0})-\tilde{u}(x_{0})=O(h).
\end{equation} 
}

\begin{rem}
Let $\zeta _{l,k+1}$ in \eqref{eq:ito-sde-vec-weak-euler} be i.i.d. random
variables with the law 
\begin{equation}
P(\zeta =\mathsf{y}_{n,j})=\mathsf{w}_{n,j},\quad j=1,\ldots ,n,
\label{eq:weak-euler-discrete-apprx-gauss-n}
\end{equation}%
where $\mathsf{y}_{n,j}$ are nodes of the Gauss-Hermite quadrature $Q_{n}$
and $\mathsf{w}_{n,j}$ are the corresponding quadrature weights (see $(\ref%
{eq:1d-Gauss-Hermite-rule})$). Then 
\begin{equation*}
\mathbb{E}f(\tilde{X}_{N})=\mathbb{E}\varphi (\zeta _{1,N},\ldots ,{}\zeta
_{r,N})=Q_{n}^{\otimes rN}\varphi (y_{1,1},\ldots ,y_{r,N}),
\end{equation*}%
which can be a more accurate approximation of $\bar{u}(x_{0})$ than $\tilde{u%
}(x_{0})$ from $(\ref{eq:ito-sde-weak-euler-functional-mean-tensor})$ but
the weak-sense error for the SDE approximation $\mathbb{E}f(\tilde{X}_{N})-%
\mathbb{E}f(X(T))$ remains of order $O(h).$
\end{rem}

Practical implementation of $\bar{u}(x_{0})$ and $\tilde{u}(x_{0})$ usually
requires the use of the Monte Carlo technique since the computational cost
of, e.g. the tensor product rule in (\ref%
{eq:ito-sde-weak-euler-functional-mean-tensor}) is prohibitively high (cf.
Section~\ref{sec:ssg}). In this paper, we consider application of the sparse
grid rule (\ref{eq:smolyak-tensor-like}) to the integral in (\ref%
{eq:strong-euler-mean-def}) motivated by lower computational cost of (\ref%
{eq:smolyak-tensor-like}).

{\color{black}In this approach, the total error has two parts
\begin{equation*}
\abs{\mathbb{E}f(X(T))- A(L,N)\varphi }\leq \abs{ \mathbb{E}f(X(T))- \mathbb{E}f(\tilde{X}_{N}) }+ \abs{\mathbb{E}f(\tilde{X}_{N}) - A(L,N)\varphi},
\end{equation*}%
where $A(L,N)$ is defined in  \eqref{eq:smolyak-tensor-like} and $\varphi$ is  from \eqref{eq:1d-sde-functional-weak}. The first part is controlled by  the  time step size $h$, see \eqref{eq:ito-sde-vec-strongEuler-error}, and it converges to zero with order one in $h$.  The second part is controlled by  the 
sparse grid level $L$ but it depends on $h$ since decrease of $h$ increases the dimension of the random space.  Some illustrative examples will be presented in Section  \ref{sec:sg-errest-illus-b}.}

\subsubsection{Probabilistic interpretation of SGC}
It is not difficult to show that SGC admits a probabilistic interpretation,
e.g. in the case of level $L=2$ we have%
\begin{eqnarray}
&&A(2,N)\varphi (y_{1,1},\ldots ,y_{r,1},\ldots ,y_{1,N},\ldots ,y_{r,N})
\label{eq:smolyak-level2-gauss-hermite-prob} \\
&=&\left( Q_{2}\otimes Q_{1}\otimes \cdots \otimes Q_{1}\right) \varphi
+\left( Q_{1}\otimes Q_{2}\otimes Q_{1}\otimes \cdots \otimes Q_{1}\right)
\varphi  \notag \\
&&+\cdots +\left( Q_{1}\otimes Q_{1}\otimes \cdots \otimes Q_{2}\right)
\varphi -(Nr-1)\left( Q_{1}\otimes Q_{1}\otimes \cdots \otimes Q_{1}\right)
\varphi  \notag \\
&=&\sum_{i=1}^{N}\sum_{j=1}^{r}\mathbb{E}\varphi (0,\ldots ,0,\zeta
_{j,i},0,\ldots ,0)-(Nr-1)\varphi (0,0,\ldots ,0),  \notag
\end{eqnarray}%
where $\zeta _{j,i}$ are i.i.d. random variables with the law (\ref%
{eq:weak-euler-discrete-apprx-gauss}). Using %
\eqref{eq:ito-sde-weak-euler-functional-mean-tensor}, %
\eqref{eq:smolyak-level2-gauss-hermite-prob}, Taylor's expansion and
symmetry of $\zeta _{j,i}$, we obtain the relationship between the weak
Euler scheme (\ref{eq:ito-sde-vec-weak-euler}) and the SGC (\ref%
{eq:smolyak-tensor-like}): 
\begin{eqnarray}
\mathbb{E}f(\tilde{X}_{N})-A(2,N)\varphi &=&\mathbb{E}\varphi (\zeta
_{1,1},\ldots ,\zeta _{r,1},\ldots ,\zeta _{1,N},\ldots ,\zeta _{r,N})
\label{eq:sg-ps-exact-remainder} \\
&&-\sum_{i=1}^{N}\sum_{j=1}^{r}\mathbb{E}\varphi (0,\ldots ,0,\zeta
_{j,i},0,\ldots ,0)-(Nr-1)\varphi (0,0,\ldots ,0)  \notag \\
&=&\sum_{\left\vert \alpha \right\vert =4}\frac{4}{\alpha !}\mathbb{E}\left[
\prod_{i=1}^{N}\prod_{j=1}^{r}(\zeta _{j,i})^{\alpha
_{j,i}}\int_{0}^{1}(1-z)^{3}D^{\alpha }\varphi (z\zeta _{1,1},\ldots ,z\zeta
_{r,N})\,dz\right]  \notag \\
&&-\frac{1}{3!}\sum_{i=1}^{N}\sum_{j=1}^{r}\mathbb{E}\left[ \zeta
_{j,i}^{4}\int_{0}^{1}(1-z)^{3}\frac{\partial ^{4}}{\left( \partial
y_{j,i}\right) ^{4}}\varphi (0,\ldots ,0,z\zeta _{j,i},0,\ldots ,0)\,dz%
\right] ,  \notag
\end{eqnarray}%
where the multi-index $\alpha =(\alpha _{1,1},\ldots ,\alpha _{r,N})\in 
\mathbb{N}_{0}^{rN}$, $\left\vert \alpha \right\vert
=\sum_{i=1}^{N}\sum_{j=1}^{r}\alpha _{j,i}$, $\alpha
!=\prod_{i=1}^{N}\prod_{j=1}^{r}\alpha _{j,i}!$ and $D^{\alpha }=\frac{%
\partial ^{|\alpha |}}{(\partial y_{1,1})^{\alpha _{1,1}}\cdots (\partial
y_{r,N})^{\alpha _{r,N}}}.$ {\color{black}Similarly, we can write down a probabilistic interpretation for any $L$ and derive a similar error representation. For example, we have for $L=3$ that}
{\color{black}
\begin{eqnarray*}
&& \mean{\varphi(\zeta_{1,1}^{(3)},\cdots, \zeta_{r,N}^{(3)})} - A(3,N) \varphi\\
&=&   \sum_{\abs{\alpha}=6} \frac{6}{\alpha !}\mean{\prod_{i=1}^{N}\prod_{j=1}^{r}(\zeta _{j,i}^{(3)})^{\alpha
_{j,i}}  \int_0^1 (1-z)^5D^\alpha  \varphi(z\zeta_{1,1}^{(3)}, \cdots,z\zeta_{r,N}^{(3)})\,d z }\\
&&-  \sum_{\substack{  \abs{\alpha}=\alpha_{j,i}+\alpha_{l,k}=6\\ (j-l)^2+(i-k)^2\neq0}} \frac{6}{\alpha_{j,i}!\alpha_{k,l}!}\mean{ (\zeta _{j,i}^{(3)})^{\alpha
_{j,i}} (\zeta _{l,k}^{(3)})^{\alpha
_{l,k}}\int_0^1 (1-z)^5 D^\alpha  \varphi(\cdots,z\zeta_{j,i}^{(3)},0,\cdots,0,z\zeta_{l,k}^{(3)},\cdots)\,dz }\\
&&-  \sum_{i=1}^N\sum_{j=1}^r\frac{6}{6 !}\mean{(\zeta_{j,i} )^6 \int_0^1 (1-z)^5 D^\alpha  \varphi(0,\cdots,z\zeta_{j,i},\cdots,0)\,dz},
\end{eqnarray*}
where $\zeta_{j,i}$ are defined in \eqref{eq:weak-euler-discrete-apprx-gauss} and  $\zeta _{j,i}^{(3)}$ are   i.i.d. random variables with the law  $P(\zeta _{j,i}^{(3)} =\pm \sqrt{3})=1/6$, $P(\zeta _{j,i}^{(3)}=0)=2/3$. 
}

The error of the SGC applied to weak-sense approximation of SDE is further studied in Section~\ref%
{sec:sg-errest-illus-b}.

\subsubsection{Second-order schemes}
In the SGC context, it is beneficial to exploit higher-order or
higher-accuracy schemes for approximating the SDE \eqref{eq:ito-sde-vector}
because they can allow us to reach a desired accuracy using larger time step
sizes and therefore less random variables than the first-order Euler scheme (%
\ref{eq:ito-sde-vec-strong-euler}) or \eqref{eq:ito-sde-vec-weak-euler}. For
instance, we can use the second-order weak scheme for %
\eqref{eq:ito-sde-vector} (see, e.g. \cite[Chapter 2]{MilTre-B04}): 
\begin{eqnarray}
X_{k+1} &=&X_{k}+ha(t_{k},X_{k})+\sqrt{h}\sum_{i=1}^{r}\sigma
_{i}(t_{k},X_{k})\xi _{i,k+1}+\frac{h^{2}}{2}\mathfrak{L}a(t_{k},X_{k})
\label{eq:ito-sde-vec-weak-2nd-order} \\
&&+h\sum_{i=1}^{r}\sum_{j=1}^{r}\Lambda _{i}\sigma _{j}(t_{k},X_{k})\eta
_{i,j,k+1}+\frac{h^{3/2}}{2}\sum_{i=1}^{r}(\Lambda _{i}a(t_{k},X_{k})+%
\mathfrak{L}\sigma _{i}(t_{k,}X_{k}))\xi _{i,k+1},  \notag \\
k &=&0,\ldots ,N-1,  \notag
\end{eqnarray}%
where $X_{0}=x_{0};$ $\eta _{i,j}=\frac{1}{2}\xi _{i}\xi _{j}-\gamma
_{i,j}\zeta _{i}\zeta _{j}/2$ with $\gamma _{i,j}=-1$ if $i<j$ and $\gamma
_{i,j}=1$ otherwise; 
\begin{equation*}
\Lambda _{l}=\sum_{i=1}^{m}\sigma _{l}^{i}\frac{\partial }{\partial x_{i}}%
,\quad \mathfrak{L}=\frac{\partial }{\partial t}+\sum_{i=1}^{m}a^{i}\frac{%
\partial }{\partial x_{i}}+\frac{1}{2}\sum_{l=1}^{r}\sum_{i,j=1}^{m}\sigma
_{l}^{i}\sigma _{l}^{i}\frac{\partial ^{2}}{\partial x_{i}\partial x_{j}};
\end{equation*}%
and $\xi _{i,k+1}$ and $\zeta _{i,k+1}$ are mutually independent random
variables with Gaussian distribution or with the laws $P(\xi =0)=2/3,\text{ }%
P(\xi =\pm \sqrt{3})=1/6\text{ \ and \ }P(\zeta =\pm 1)=1/2.$ The following
error estimate holds for (\ref{eq:ito-sde-vec-weak-2nd-order}) (see e.g. 
\cite[Chapter 2]{MilTre-B04}): 
\begin{equation*}
\left\vert \mathbb{E}f(X(T))-Ef(X_{N})\right\vert \leq Kh^{2}.
\end{equation*}

Roughly speaking, to achieve $O(h)$ accuracy using %
\eqref{eq:ito-sde-vec-weak-2nd-order}, we need only $\sqrt{2rN}$ ($\sqrt{rN}$
in the case of additive noise) random variables, while we need $rN$ random
variables for the Euler scheme \eqref{eq:ito-sde-vec-strong-euler}. This
reduces the dimension of the random space and hence can increase efficiency
and widen applicability of SGC methods (see, in particular Example~4.1 in
Section~\ref{sec:num-experiments-sg}\ for a numerical illustration). We note
that when noise intensity is relatively small, we can use high-accuracy
low-order schemes designed for SDE with small noise \cite{MilTre97} (see
also \cite[Chapter 3]{MilTre-B04}) in order to achieve a desired accuracy
using less number of random variables than the Euler scheme %
\eqref{eq:ito-sde-vec-strong-euler}.

\subsection{Illustrative examples\label{sec:sg-errest-illus-b}}

In this section we show limitations of the use of SGC in weak approximation
of SDE. To this end, it is convenient and sufficient to consider the scalar
linear SDE 
\begin{equation}
dX=\lambda Xdt+\varepsilon \,dw(t),\quad X_{0}=1,  \label{linearSDE}
\end{equation}%
where $\lambda \ $and $\varepsilon $ are some constants.

We will compute expectations $\mathbb{E}f(X(T))$ for some $f(x)$ and $X(t)$
from (\ref{linearSDE}) by applying the Euler scheme (\ref%
{eq:ito-sde-vec-strong-euler}) and the SGC (\ref{eq:smolyak-tensor-like}).
This simple example provides us with a clear insight when algorithms of this
type are able to produce accurate results and when they are likely to fail.
Using direct calculations, we first (see Examples~\ref%
{exm:linear-sde-moments}--\ref{exm:linear-sde-cos} below) derive an estimate
for the error $\left\vert \mathbb{E}f(X_{N})-A(2,N)\varphi \right\vert $
with $X_{N}$ from (\ref{eq:ito-sde-vec-strong-euler}) applied to (\ref%
{linearSDE}) and for some particular $f(x)$.  {\color{black}This will illustrate how the error of SGC with 
practical level  (no more than six)  behaves.}
Then (Proposition~\ref%
{prop:weak-apprx-euler-sg}) we obtain an estimate for the error $\left\vert 
\mathbb{E}f(X_{N})-A(L,N)\varphi \right\vert $ for a smooth $f(x)$ which
grows not faster than a polynomial function at infinity. We will observe
that the considered algorithm is not convergent in time step $h$ and {\color{black}   the algorithm 
is not convergent in level $L$ unless when noise
intensity and integration time are small.}

It follows from \eqref{eq:ito-sde-vec-strongEuler-error} and %
\eqref{eq:ito-sde-vec-weakEuler-error} that 
\begin{eqnarray}
\left\vert \mathbb{E}f(X_{N})-A(L,N)\varphi \right\vert &\leq &\left\vert 
\mathbb{E}f(\tilde{X}_{N})-A(L,N)\varphi \right\vert +|\mathbb{E}f(X_{N})-%
\mathbb{E}f(\tilde{X}_{N})|  \label{eq:simple-exam-error} \\
&\leq &\left\vert \mathbb{E}f(\tilde{X}_{N})-A(L,N)\varphi \right\vert +Kh, 
\notag
\end{eqnarray}%
where $\tilde{X}_{N}$ is from the weak Euler scheme %
\eqref{eq:ito-sde-vec-weak-euler} applied to (\ref{linearSDE}), which can be
written as $\tilde{X}_{N}=\displaystyle(1+\lambda
h)^{N}+\sum_{j=1}^{N}(1+\lambda h)^{N-j}\varepsilon \sqrt{h}\zeta _{j}.$
Introducing the function 
\begin{equation}
\bar{X}(N;y)=(1+\lambda h)^{N}+\sum_{j=1}^{N}(1+\lambda h)^{N-j}\varepsilon 
\sqrt{h}y_{j},  \label{Xy}
\end{equation}%
we see that $\tilde{X}_{N}=\bar{X}(N;\zeta _{1},\ldots ,\zeta _{N}).$ We
have 
\begin{equation}
\frac{\partial }{\partial y_{i}}\bar{X}(N;y)=(1+\lambda h)^{N-i}\varepsilon 
\sqrt{h}\text{\ \ and\ \ }\frac{\partial ^{2}}{\partial y_{i}\partial y_{j}}%
\bar{X}(N;y)=0.  \label{Xder}
\end{equation}%
Then we obtain from (\ref{eq:sg-ps-exact-remainder}): 
\begin{eqnarray}
R &:&=\mathbb{E}f(\tilde{X}_{N})-A(2,N)\varphi  \label{Rnew} \\
&=&\varepsilon ^{4}h^{2}\sum_{\left\vert \alpha \right\vert =4}\frac{4}{%
\alpha !}\mathbb{E}\left[ \prod_{i=1}^{N}(\zeta _{i}(1+\lambda
h)^{N-i})^{\alpha _{i}}\int_{0}^{1}(1-z)^{3}z^{4}\frac{d^{4}}{dx^{4}}f(\bar{X%
}(N,z\zeta _{1},\ldots ,z\zeta _{N}))\,dz\right]  \notag \\
&&-\frac{1}{3!}\varepsilon ^{4}h^{2}\sum_{i=1}^{N}\mathbb{E}\left[ \zeta
_{i}^{4}\int_{0}^{1}(1-z)^{3}z^{4}\frac{d^{4}}{dx^{4}}f(\bar{X}(0,\ldots
,0,z\zeta _{i},0,\ldots ,0))\,(1+\lambda h)^{4N-4i}dz\right] .  \notag
\end{eqnarray}

\subsubsection{Non-Convergence in time step $h$}

We will illustrate  {\color{black}no convergence} in $h$ {\color{black}for SGC  for  levels  two and three through two examples, 
where sharp error estimates of $\abs{\mathbb{E}f(X_{N})-A(2,N)\varphi}$ are derived for  SGC. Higher level SGC can be also considered but the conclusion  do  not change. In contrast,    
the algorithm of tensor-product integration in random space and the strong Euler scheme in time  
(i.e., the weak Euler scheme  \eqref{eq:ito-sde-vec-weak-euler}-\eqref{eq:weak-euler-discrete-apprx-gauss})
is convergent with order one in $h$.  We also note that in practice,  typically  SGC with level no more than six are employed.}
\begin{exm}
\label{exm:linear-sde-moments}%
\upshape
For $f(x)=x^{p}$ with $p=1,2,3$, it follows from (\ref{Rnew}) that $R=0,$
i.e., SGC does not introduce any additional error, and hence by %
\eqref{eq:simple-exam-error} 
\begin{equation*}
\left\vert \mathbb{E}f(X_{N})-A(2,N)\varphi \right\vert \leq Kh,\quad
f(x)=x^{p},\quad p=1,2,3.
\end{equation*}%
For $f(x)=x^{4}$, we get from \eqref{Rnew}: 
\begin{eqnarray*}
R &=&\frac{6}{35}\varepsilon
^{4}h^{2}\sum_{i=1}^{N}\sum_{j=i+1}^{N}(1+\lambda h)^{4N-2i-2j} \\
&=&\frac{6}{35}\varepsilon ^{4}\times \left\{ 
\begin{tabular}{ll}
$\frac{(1+\lambda h)^{2N}-1}{\lambda ^{2}(2+\lambda h)^{2}}\left[ \frac{%
(1+\lambda h)^{2N}+1}{1+(1+\lambda h)^{2}}-1\right] ,$ & $\lambda \neq 0,\ \
1+\lambda h\neq 0,$ \\ 
$\frac{T^{2}}{2}-\frac{Th}{2},$ & $\lambda =0.$%
\end{tabular}%
\right.
\end{eqnarray*}%
We see that $R$ does not go to zero when $h\rightarrow 0$ and that for
sufficiently small $h>0$ 
\begin{equation*}
\left\vert \mathbb{E}f(X_{N})-A(2,N)\varphi \right\vert \leq Kh+\frac{6}{35}%
\varepsilon ^{4}\times \left\{ 
\begin{tabular}{ll}
$\frac{1}{\lambda ^{2}}(1+e^{4T\lambda }),$ & $\lambda \neq 0,$ \\ 
$\frac{T^{2}}{2},$ & $\lambda =0.$%
\end{tabular}%
\right.
\end{equation*}
\end{exm}

We observe that the SGC algorithm does not converge with $h\rightarrow 0$
for higher moments. In the considered case of linear SDE, increasing the
level $L$ of SGC leads to the SGC error $R$ being $0$ for higher moments,
e.g., for $L=3$ the error $R=0$ for up to $5$th moment but the algorithm
will not converge in $h$ for $6$th moment and so on (see Proposition~\ref%
{prop:weak-apprx-euler-sg} below). Further (see the continuation of the
illustration below), in the case of, e.g. $f(x)=\cos x$ for any $L$ this
error $R$ {\color{black} does not converge in $h$}, which is also the case for nonlinear SDE. We also
note that one can expect that this error $R$ is small when noise intensity
is relatively small and either time $T$ is small or SDE has, in some sense,
stable behavior (in the linear case it corresponds to $\lambda <0).$

\begin{exm}
\label{exm:linear-sde-cos}%
\upshape%
Now consider $f(x)=\cos (x).$ It follows from \eqref{Rnew} that 
\begin{eqnarray*}
R &=&\varepsilon ^{4}h^{2}\sum_{\left\vert \alpha \right\vert =4}\frac{4}{%
\alpha !}\mathbb{E}\left[ \prod_{i=1}^{N}(\zeta _{i}(1+\lambda
h)^{N-i})^{\alpha _{i}}\int_{0}^{1}(1-z)^{3}z^{4}\cos ((1+\lambda
h)^{N}\right. \\
&&\left. +z\sum_{j=1}^{N}(1+\lambda h)^{N-j}\varepsilon \sqrt{h}\zeta
_{j})\,dz\right] \\
&&-\frac{1}{3!}\varepsilon ^{4}h^{2}\sum_{i=1}^{N}(1+\lambda
h)^{4N-4i}\int_{0}^{1}(1-z)^{3}z^{4}\mathbb{E}[\zeta _{i}^{4}\cos
((1+\lambda h)^{N}+z(1+\lambda h)^{N-i}\varepsilon \sqrt{h}\zeta _{i})]\,dz
\end{eqnarray*}%
and after routine calculations we obtain 
\begin{eqnarray*}
R &=&\varepsilon ^{4}h^{2}\cos ((1+\lambda h)^{N})\left[ \left( \frac{1}{6}%
\sum_{i=1}^{N}(1+\lambda
h)^{4N-4i}+2\sum_{i=1}^{N}\sum_{j=i+1}^{N}(1+\lambda h)^{4N-2i-2j}\right)
\right. \\
&&\times \int_{0}^{1}(1-z)^{3}z^{4}\prod_{l=1}^{N}\cos (z(1+\lambda
h)^{N-l}\varepsilon \sqrt{h})dz\\
&&+\left( \frac{2}{3}\sum_{i,j=1;i\neq j}^{N}(1+\lambda h)^{4N-3i-j}+2\sum 
_{\substack{ k,i,j=1  \\ i\neq j,i\neq k,k\neq j}}^{N}(1+\lambda
h)^{4N-2k-i-j}\right) \\
&&\times \int_{0}^{1}(1-z)^{3}z^{4}\prod_{l=i,j}\sin (z(1+\lambda
h)^{N-l}\varepsilon \sqrt{h})\prod_{\substack{ l=1  \\ l\neq i,l\neq j}}%
^{N}\cos (z(1+\lambda h)^{N-l}\varepsilon \sqrt{h})dz \\
&&+4\sum_{\substack{ i,j,k,m=1  \\ i\neq j,i\neq k,i\neq m,j\neq k,j\neq
m,k\neq m}}^{N}(1+\lambda h)^{4N-i-j-k-m} \\
&&\times \int_{0}^{1}(1-z)^{3}z^{4}\prod_{l=i,j,{\color{black}k},m}\sin (z(1+\lambda
h)^{N-l}\varepsilon \sqrt{h})\prod_{\substack{ l=1  \\ l\neq i,l\neq j,l\neq
k,l\neq m}}^{N}\cos (z(1+\lambda h)^{N-l}\varepsilon \sqrt{h})dz \\
&&\left. -\frac{1}{6}\sum_{i=1}^{N}(1+\lambda
h)^{4N-4i}\int_{0}^{1}(1-z)^{3}z^{4}\cos (z(1+\lambda h)^{N-i}\varepsilon 
\sqrt{h})]\,dz\right] .
\end{eqnarray*}%
It is not difficult to see that $R$ does not go to zero when $h\rightarrow
0. $  {\color{black}In fact,} taking into account that $|\sin (z(1+\lambda
h)^{N-j}\varepsilon \sqrt{h})|\leq z(1+\lambda h)^{N-j}\varepsilon \sqrt{h},$ {\color{black} and 
that there are $N^4$   terms of order $h^4$ and $N^3$ terms of  order $h^3$},
we get for sufficiently small $h>0$%
\begin{equation*}
|R|\leq C\varepsilon ^{4} (1+e^{4T\lambda }),
\end{equation*}%
where $C>0$ is independent of $\varepsilon $ and $h.$ Hence 
\begin{equation}
\left\vert \mathbb{E}f(X_{N})-A(2,N)\varphi \right\vert \leq C\varepsilon
^{4}(1+e^{4T\lambda })+Kh,  \label{eq:estimate-euler-cos}
\end{equation}%
and we have arrived at a similar conclusion for $f(x)=\cos x$ as for $%
f(x)=x^{4}$. {\color{black}Similarly, we can also have for $L=3$ that
\begin{equation*}
 \abs{\mathbb{E}f(X_{N})-A(3,N)\varphi} \leq C\varepsilon
^{6}(1+e^{6T\lambda })+Kh.
\end{equation*}
This example shows for $L=3$, the error of SGC with the Euler scheme in time does not converge in $h$.}
\end{exm}

\subsubsection{\color{black} Error estimate for SGC with fixed level}
{\color{black} 
Now we will address  the effect  of the SGC level $L$.} To this end, we will need the following error
estimate of a Gauss-Hermite quadrature. Let $\psi (y),$ $y\in \mathbb{R}$,
be a sufficiently smooth function which itself and its derivatives are
growing not faster than a polynomial at infinity. Using the Peano kernel
theorem (see e.g. \cite{DavRab-B84}) and that a Gauss-Hermite quadrature
with $n$-nodes has the order of polynomial exactness $2n-1,$ we obtain for
the approximation error $R_{n,\gamma }\psi $ of the Gauss-Hermite quadrature 
$Q_{n}\psi $: 
\begin{equation}
R_{n,\gamma }(\psi ):=Q_{n}\psi -I_{1}\psi =\int_{\mathbb{R}}\frac{d^{\gamma
}}{dy^{\gamma }}\varphi (y)R_{n,\gamma }(\Gamma _{y,\gamma })\,dy,\quad
1\leq \gamma \leq 2n,  \label{eq:gaussHer-quad-remainder}
\end{equation}%
where $\Gamma _{y,\gamma }(z)=(z-y)^{\gamma -1}/(\gamma -1)!$ if $z\geq y$
and $0$ otherwise. One can show (see, e.g. \cite[Theorem 2]{MasMon94}) that
there is a constant $c>0$ independent of $n$ and $y$ such that for any $%
0<\beta <1$%
\begin{equation}
\left\vert R_{n,\gamma }(\Gamma _{y,\gamma })\right\vert \leq \frac{c}{\sqrt{%
2\pi }}n^{-\gamma /2}\exp \left( -\frac{\beta y^{2}}{2}\right) ,\quad 1\leq
\gamma \leq 2n.  \label{eq:estimate-GaussHer-peano-kernel}
\end{equation}%
We also note that (\ref{eq:estimate-GaussHer-peano-kernel}) and the triangle
inequality imply, for $1\leq \gamma \leq 2(n-1)$: 
\begin{equation}
\left\vert R_{n,\gamma }(\Gamma _{y,\gamma })-R_{n-1,\gamma }(\Gamma
_{y,\gamma })\right\vert \leq \frac{c}{\sqrt{2\pi }}[n^{-\gamma
/2}+(n-1)^{-\gamma /2}]\exp \left( -\frac{\beta y^{2}}{2}\right) .
\label{eq:estimate-GaussHer-peano-kernel-diff}
\end{equation}

Now we consider an error of the sparse grid rule %
\eqref{eq:smolyak-tensor-like} accompanied by the Euler scheme %
\eqref{eq:ito-sde-vec-strong-euler} for computing expectations of solutions
to \eqref{linearSDE}.

\begin{prop}
\label{prop:weak-apprx-euler-sg} Assume that a function $f(x)$ and its
derivatives up to $2L$-th order satisfy the polynomial growth condition $(%
\ref{f_cond})$. Let $X_{N}$ be obtained by the Euler scheme %
\eqref{eq:ito-sde-vec-strong-euler} applied to the linear SDE $(\ref%
{linearSDE})$ and $A(L,N)\varphi $ be the sparse grid rule $(\ref%
{eq:smolyak-tensor-like})$ with level $L$ applied to the integral
corresponding to $\mathbb{E}f(X_{N})$ as in $(\ref{eq:strong-euler-mean-def}%
) $. Then for {\color{black} any $L$} and sufficiently small $h>0$ 
\begin{equation}  \label{propstat}
\left\vert \mathbb{E}f(X_{N})-A(L,N)\varphi \right\vert \leq K\varepsilon
^{2L}(1+e^{\lambda (2L+\varkappa )T}){\color{black}\left( 1+\left( 3c/2\right) ^{L\wedge N}\right)
\beta ^{-(L\wedge N)/2} }T^{L},
\end{equation}%
where $K>0$ is independent of $h,$ $L$ and $N$; $c$ and $\beta $ are from %
\eqref{eq:estimate-GaussHer-peano-kernel}; $\varkappa $ is from %
\eqref{f_cond}.
\end{prop}

\begin{proof}
We recall (see (\ref{eq:strong-euler-mean-def})) that 
\begin{equation*}
\mathbb{E}f(X_{N})=I_{N}\varphi =\frac{1}{(2\pi )^{N/2}}\int_{\mathbb{R}%
^{rN}}\varphi (y_{1},\ldots ,y_{N})\exp \left( {-\frac{1}{2}%
\sum_{i=1}^{N}y_{i}^{2}}\right) \,dy.
\end{equation*}%
Introduce the integrals 
\begin{equation}
I_{1}^{(k)}\varphi =\frac{1}{\sqrt{2\pi }}\int_{\mathbb{R}}\varphi
(y_{1},\ldots ,y_{k},\ldots ,y_{N})\exp \left( {-\frac{{y_{k}^{2}}}{2}}%
\right) dy_{k},\ \ k=1,\ldots ,N,  \label{Ik1}
\end{equation}%
and their approximations $Q_{n}^{(k)}$ by the corresponding one-dimensional
Gauss-Hermite quadratures with $n$ nodes. Also, let $\mathcal{U}%
_{i_{k}}^{(k)}=Q_{i_{k}}^{(k)}-Q_{i_{k}-1}^{(k)}.$

Using (\ref{eq:smolyak-tensor-like}) and the recipe from the proof of
Lemma~3.4 in \cite{NobTW08}, we obtain 
\begin{equation}
I_{N}\varphi -A(L,N)\varphi =\sum_{l=2}^{N}S(L,l)\otimes _{k=l+1}^{N}{I}%
_{1}^{(k)}\varphi +({I}_{1}^{(1)}-Q_{L}^{(1)})\otimes _{k=2}^{N}{I}%
_{1}^{(k)}\varphi ,  \label{eq34}
\end{equation}%
where 
\begin{equation}
S(L,l)=\sum_{i_{1}+\cdots +i_{l-1}+i_{l}=L+l-1}\otimes _{k=1}^{l-1}\mathcal{U%
}_{i_{k}}^{(k)}\otimes ({I}_{1}^{(l)}-Q_{i_{l}}^{(l)}).
\label{eq:smolyak-int-recursive-component}
\end{equation}

Due to (\ref{eq:gaussHer-quad-remainder}), we have for $n>1$ and $1\leq
\gamma \leq 2(n-1)$ 
\begin{eqnarray}
\mathcal{U}_{n}\psi &=&Q_{n}\psi -Q_{n-1}\psi =[Q_{n}\psi -I_{1}(\psi
)]-[Q_{n-1}\psi -I_{1}(\psi )]  \label{eq:gaussHer-quad-difference} \\
&=&\int_{\mathbb{R}}\frac{d^{\gamma }}{dy^{\gamma }}\psi (y)[R_{n,\gamma
}(\Gamma _{y,\gamma })-R_{n-1,\gamma }(\Gamma _{y,\gamma })]\,dy,  \notag
\end{eqnarray}%
and for $n=1$ 
\begin{equation}
\mathcal{U}_{n}\psi =Q_{1}\psi -Q_{0}\psi =Q_{1}\psi =\psi (0).
\label{eq:gaussHer-quad-differencen=1}
\end{equation}%
By \eqref{eq:smolyak-int-recursive-component}, (\ref{Ik1}) and %
\eqref{eq:gaussHer-quad-remainder}, we obtain for the first term in the
right-hand side of (\ref{eq34}): 
\begin{eqnarray*}
 S(L,l)\otimes _{n=l+1}^{N}I_{1}^{(n)}\varphi  
&=&\sum_{i_{1}+\cdots +i_{l}=L+l-1}\otimes _{n=1}^{l-1}\mathcal{U}%
_{i_{k}}^{(k)}\otimes ({I}_{1}^{(l)}-Q_{i_{l}}^{(l)})\otimes
_{n=l+1}^{N}I_{1}^{(n)}\varphi \\
&=&\sum_{i_{1}+\cdots +i_{l}=L+l-1}\otimes _{n=1}^{l-1}\mathcal{U}%
_{i_{k}}^{(k)}\otimes ({I}_{1}^{(l)}-Q_{i_{l}}^{(l)}) \\
&&\otimes \int_{\mathbb{R}^{N-l}}\varphi (y)\frac{1}{(2\pi )^{(N-l)/2}}\exp
(-\sum_{k=l+1}^{N}\frac{y_{k}^{2}}{2})\,dy_{l+1}\ldots dy_{N} \\
&=&-\sum_{i_{1}+\cdots +i_{l}=L+l-1}\otimes _{n=1}^{l-1}\mathcal{U}%
_{i_{k}}^{(k)}\otimes \int_{\mathbb{R}^{N-l+1}}\frac{d^{2i_{l}}}{%
dy_{l}^{2i_{l}}}\varphi (y)R_{i_{l},2i_{l}}(\Gamma _{y_{l},2i_{l}}) \\
&&\times \frac{1}{(2\pi )^{(N-l)/2}}\exp (-\sum_{k=l+1}^{N}\frac{y_{k}^{2}}{2%
})\,dy_{l}\ldots dy_{N}.
\end{eqnarray*}%
Now consider two cases: if $i_{l-1}>1$ then by %
\eqref{eq:gaussHer-quad-difference}:%
\begin{eqnarray*}
S(L,l)\otimes _{n=l+1}^{N}I_{1}^{(n)}\varphi &=&-\sum_{i_{1}+\cdots
+i_{l}=L+l-1}\otimes _{n=1}^{l-2}\mathcal{U}_{i_{k}}^{(k)}\otimes \int_{%
\mathbb{R}^{N-l+2}}\frac{d^{2i_{l-1}-2}}{dy_{l-1}^{2i_{l-1}-2}}\frac{%
d^{2i_{l}}}{dy_{l}^{2i_{l}}}\varphi (y)R_{i_{l},2i_{l}}(\Gamma
_{y_{l},2i_{l}}) \\
&&\times \lbrack R_{i_{l-1},2i_{l-1}-2}(\Gamma
_{y_{l-1},2i_{l-1}-2})-R_{i_{l-1}-1,2i_{l-1}-2}(\Gamma
_{y_{i_{l-1}},2i_{l-1}-2})] \\
&&\times \frac{1}{(2\pi )^{(N-l)/2}}\exp (-\sum_{k=l+1}^{N}\frac{y_{k}^{2}}{2%
})\,dy_{l-1}\ldots \,dy_{N}
\end{eqnarray*}%
otherwise (i.e., if $i_{l-1}=1)$ by \eqref{eq:gaussHer-quad-differencen=1}: 
\begin{eqnarray*}
S(L,l)\otimes _{n=l+1}^{N}I_{1}^{(n)}\varphi &=&-\sum_{i_{1}+\cdots
+i_{l}=L+l-1}\otimes _{n=1}^{l-2}\mathcal{U}_{i_{k}}^{(k)}\otimes \int_{%
\mathbb{R}^{N-l+1}}Q_{1}^{(l-1)}\frac{d^{2i_{l}}}{dy_{l}^{2i_{l}}}\varphi
(y)R_{i_{l},2i_{l}}(\Gamma _{y_{l},2i_{l}}) \\
&&\times \frac{1}{(2\pi )^{(N-l)/2}}\exp (-\sum_{k=l+1}^{N}\frac{y_{k}^{2}}{2%
})\,dy_{l}\ldots dy_{N}.
\end{eqnarray*}%
Repeating the above process for $i_{l-2},\ldots ,i_{1}$, we obtain 
\begin{gather}
S(L,l)\otimes _{n=l+1}^{N}I_{1}^{(n)}\varphi =\sum_{i_{1}+\cdots
+i_{l}=L+l-1}\int_{\mathbb{R}^{N-\#F_{l-1}}}[\otimes _{m\in
F_{l-1}}Q_{1}^{(m)}D^{2\alpha _{l}}\varphi (y)]
\label{eq:sg-component-recursive-l} \\
\times \mathcal{R}_{l,\alpha _{l}}(y_{1},\ldots ,y_{l})\frac{1}{(2\pi
)^{(N-l)/2}}\exp (-\sum_{k=l+1}^{N}\frac{y_{k}^{2}}{2})\prod_{n\in
G_{l-1}}\,dy_{n}\times \,dy_{l}\ldots dy_{N},  \notag
\end{gather}%
where the multi-index $\alpha _{l}=(i_{1}-1,\ldots ,i_{l-1}-1,i_{l},0,\ldots
,0)$ with the $m$-th element $\alpha _{l}^{m},$ the sets $%
F_{l-1}=F_{l-1}(\alpha _{l})=\left\{ m:~\alpha _{l}^{m}=0,\text{ }m=1,\ldots
,l-1\right\} $ and $G_{l-1}=G_{l-1}(\alpha _{l})=\left\{ m:~\alpha
_{l}^{m}>0,\text{ }m=1,\ldots ,l-1\right\} $, the symbols $\#F_{l-1}$ and $%
\#G_{l-1}$ stand for the number of elements in the corresponding sets, and 
 \begin{eqnarray*}
 \mathcal{R}_{l,\alpha _{l}}(y_{1},\ldots ,y_{l})=-R_{i_{l},2i_{l}}(\Gamma
_{y_{l},2i_{l}})\otimes _{n\in G_{l-1}}[R_{i_{n},2i_{n}-2}(\Gamma
_{y_{n},2i_{n}-2})-R_{i_{n}-1,2i_{n}-2}(\Gamma _{y_{n},2i_{n}-2})].
 \end{eqnarray*}
 
Note that $\#G_{l-1}\leq (L-1)\wedge (l-1)$ and also recall that $i_{j}\geq
1,$ $j=1,\ldots ,l.$

Using \eqref{eq:estimate-GaussHer-peano-kernel}, %
\eqref{eq:estimate-GaussHer-peano-kernel-diff} and the inequality 
\begin{equation*}
\prod_{n\in
G_{l-1}}[i_{n}^{-(i_{n}-1)}+(i_{n}-1)^{-(i_{n}-1)}]i_{l}^{-i_{l}}\leq
(3/2)^{\#G_{l-1}},
\end{equation*}%
we get 
\begin{eqnarray}
\left\vert \mathcal{R}_{l,\alpha }(y_{1},\ldots ,y_{l})\right\vert &\leq
&\prod_{n\in
G_{l-1}}[i_{n}^{-(i_{n}-1)}+(i_{n}-1)^{-(i_{n}-1)}]i_{l}^{-i_{l}}\frac{%
c^{\#G_{l-1}+1}}{(2\pi )^{(\#G_{l-1}+1)/2}}\ \ 
\label{eq:err-est-multi-peano-kernel} \\
&&\times \exp \left( -\sum_{n\in G_{l-1}}\frac{\beta y_{n}^{2}}{2}-\frac{%
\beta y_{l}^{2}}{2}\right)  \notag \\
&\leq &\frac{(3c/2)^{\#G_{l-1}+1}}{(2\pi )^{(\#G_{l-1}+1)/2}}\exp \left(
-\sum_{n\in G_{l-1}}\frac{\beta y_{n}^{2}}{2}-\frac{\beta y_{l}^{2}}{2}%
\right) .  \notag
\end{eqnarray}%
Substituting \eqref{eq:err-est-multi-peano-kernel} in %
\eqref{eq:sg-component-recursive-l}, we arrive at 
\begin{eqnarray}
&&\left\vert S(L,l)\otimes _{n=l+1}^{N}I_{1}^{(n)}\varphi \right\vert
\label{eq:sg-error-estimates-components-est} \\
&\leq &\sum_{i_{1}+\cdots +i_{l}=L+l-1}\frac{(3c/2)^{\#G_{l-1}+1}}{(2\pi
)^{(N-\#F_{l-1})/2}}\int_{\mathbb{R}^{N-\#F_{l-1}}}\left\vert \otimes _{m\in
F_{l-1}}Q_{1}^{(m)}D^{2\alpha _{l}}\varphi (y)\right\vert  \notag \\
&&\times \exp \left( -\sum_{n\in G_{l-1}}\frac{\beta y_{n}^{2}}{2}-\frac{%
\beta y_{l}^{2}}{2}-\sum_{k=l+1}^{N}\frac{y_{k}^{2}}{2}\right) \prod_{n\in
G_{l-1}}dy_{n}\times \,dy_{l}\ldots dy_{N}.  \notag
\end{eqnarray}%
Using (\ref{Xder}) and the assumption that $\left\vert \frac{d^{2L}}{dx^{2L}}%
f(x)\right\vert \leq K(1+|x|^{\varkappa })$ for some $K>0$ and $\varkappa
\geq 1$, we get 
\begin{eqnarray}
\left\vert D^{2\alpha _{l}}\varphi (y)\right\vert &=&\varepsilon
^{2L}h^{L}\left\vert \frac{d^{2L}}{dx^{2L}}f(\bar{X}(N,y))\right\vert
(1+\lambda h)^{2LN-2\sum_{i=1}^{l}i\alpha _{l}^{i}}  \label{Dfi} \\
&\leq &K\varepsilon ^{2L}h^{L}(1+\lambda h)^{2LN-2\sum_{i=1}^{l}i\alpha
_{l}^{i}}(1+|\bar{X}(N,y)|^{\varkappa }).  \notag
\end{eqnarray}%
Substituting (\ref{Dfi}) and (\ref{Xy}) in (\ref%
{eq:sg-error-estimates-components-est}) and doing further calculations, we
obtain 
\begin{gather}
\left\vert S(L,l)\otimes _{n=l+1}^{N}I_{1}^{(n)}\varphi \right\vert \leq
K\varepsilon ^{2L}h^{L}(1+e^{\lambda \varkappa T})(1+(3c/2)^{L\wedge
l})\beta ^{-(L\wedge l)/2}  \label{fff} \\
\times \sum_{i_{1}+\cdots +i_{l}=L+l-1}(1+\lambda
h)^{2LN-2\sum_{i=1}^{l}i\alpha _{l}^{i}}  \notag \\
\leq K\varepsilon ^{2L}h^{L}(1+e^{\lambda (2L+\varkappa
)T})(1+(3c/2)^{L\wedge l})\beta ^{-(L\wedge l)/2}\binom{L+l-2}{L-1}  \notag
\\
\leq K\varepsilon ^{2L}h^{L}(1+e^{\lambda (2L+\varkappa
)T})(1+(3c/2)^{L\wedge l})\beta ^{-(L\wedge l)/2}l^{L-1}.  \notag
\end{gather}%
with a new $K>0$ which does not depend on $h,$ $\varepsilon ,$ $L,$ $c,$ $%
\beta ,$ and $l.$ In the last line of (\ref{fff}) we used 
\begin{equation*}
\binom{L+l-2}{L-1}=\prod_{i=1}^{L-1}(1+\frac{l-1}{i})\leq \left[ \frac{1}{L-1%
}\sum_{i=1}^{L-1}(1+\frac{l-1}{i})\right] ^{L-1}\leq l^{L-1}.
\end{equation*}%
Substituting (\ref{fff}) in (\ref{eq34}) and observing that $\left\vert ({I}%
_{1}^{(1)}-Q_{L}^{(1)})\otimes _{k=2}^{N}{I}_{1}^{(k)}\varphi \right\vert $
is of order $O(h^{L})$, we arrive at (\ref{propstat}).
\end{proof}

\begin{rem}\label{rem:convergence-sgc}
Due to Examples~\ref{exm:linear-sde-moments} and \ref{exm:linear-sde-cos},
the error estimate $(\ref{propstat})$ proved in Proposition~\ref%
{prop:weak-apprx-euler-sg} is quite sharp and we conclude that in general
the SGC algorithm for weak approximation of SDE does not converge with
neither decrease of time step $h$ nor with increase of the level $L$. At the
same time, {\color{black} the algorithm is convergent in $L$ (when $L\leq N$)   if $\varepsilon^2T$ is  sufficiently small and   SDE has some stable behavior (e.g., $\lambda\leq0$). Furthermore,} the algorithm 
is   sufficiently accurate when noise intensity $%
\varepsilon $ and integration time $T$ are relatively small.
\end{rem}

\begin{rem}
It follows from the proof (see $(\ref{Dfi})$) that if $\frac{d^{2L}}{dx^{2L}}%
f(x)=0$ then the error $I_{N}(\varphi )-A(L,N)\varphi =0.$ We emphasize that
this is a feature of the linear SDE $(\ref{linearSDE})$ thanks to $(\ref%
{Xder})$, while in the case of nonlinear SDE this error remains of the form $%
(\ref{propstat})$ even if the $2L$th derivative of $f$ is zero. See also the
discussion at the end of Example \ref{exm:linear-sde-moments} and numerical
tests in Example \ref{exm:mcir}.
\end{rem}

\begin{rem}
We note that it is possible to prove a proposition analogous to Proposition~%
\ref{prop:weak-apprx-euler-sg} for a more general SDE, e.g. for SDE with
additive noise. Since such a proposition does not add further information to
our discussion of the use of SGC and its proof is more complex than in the
case of $(\ref{linearSDE})$, we do not consider such a proposition here.
\end{rem}

\section{Recursive collocation algorithm for linear SPDE}
\label{sec:recursive-sg-advdiff}

In the previous section we have demonstrated the limitations of SGC
algorithms in application to SDE that, in general, such an algorithm will
not work unless integration time $T$ and magnitude of noise are small. It is
not difficult to understand that SGC algorithms have the same limitations in
the case of SPDE as well, which, in particular, is demonstrated in
Example~4.2, where a stochastic Burgers equation is considered. To cure this
deficiency and achieve longer time integration in the case of linear SPDE,
we will exploit the idea of the recursive approach proposed in \cite%
{LotMR97,ZhangRTK12} in the case of a Wiener chaos expansion method. To this
end, we apply the algorithm of SGC accompanied by a time discretization of
SPDE over a small interval $[(k-1)h,kh]$ instead of the whole interval $%
[0,T] $ as we did in the previous section and build a recursive scheme to
compute the second-order moments of the solutions to linear SPDE.

Consider the following linear SPDE in Ito's form: 
\begin{eqnarray}
d{u(t,x)} &=&\left[ \mathcal{L}u(t,x)+f(x)\right] dt+\sum_{l=1}^{r}\left[ 
\mathcal{M}_{l}u(t,x)+g_{l}(x)\right] {d}w_{l}(t),\;(t,x)\in (0,T]\times 
\mathcal{D},\ \ \ \ \ \ \ \   \label{eq:sadv-diff} \\
{u(0,x)} &=&u_{0}(x),\ \ x\in \mathcal{D},  \notag
\end{eqnarray}%
where $\mathcal{D}$ is an open domain in $\mathbb{R}^{m}$ and $(w(t),%
\mathcal{F}_{t})$ is a Wiener process as in \eqref{eq:ito-sde-vector}, and 
\begin{eqnarray}
\mathcal{L}u(t,x) &=&\sum_{i,j=1}^{m}a_{ij}\left( x\right) \frac{\partial
^{2}}{\partial x_{i}\partial x_{j}}u(t,x)+\sum_{i=1}^{m}b_{i}(x)\frac{%
\partial }{\partial x_{i}}u(t,x)+c\left( x\right) u(t,x),
\label{eq:sadv-diff-coefficients} \\
\mathcal{M}_{l}u(t,x) &=&\sum_{i=1}^{m}\alpha _{i}^{l}(x)\frac{\partial }{%
\partial x_{i}}u(t,x)+\beta ^{l}\left( x\right) u(t,x).  \notag
\end{eqnarray}%
We assume that $\mathcal{D}$ is either bounded with regular boundary or that 
$\mathcal{D}=\mathbb{R}^{m}.$ In the former case we consider periodic
boundary conditions and in the latter the Cauchy problem. We also assume
that the coefficients of the operators $\mathcal{L}$ and $\mathcal{M}$ are
uniformly bounded and 
\begin{equation*}
\tilde{\mathcal{L}}:=\mathcal{L}-\frac{1}{2}\sum_{1\leq l\leq r}\,\mathcal{M}%
_{l}\mathcal{M}_{l}
\end{equation*}%
is nonnegative definite. When the coefficients of $\mathcal{L}$ and $%
\mathcal{M}$ are sufficiently smooth, existence and uniqueness results for
the solution of \eqref{eq:sadv-diff}-\eqref{eq:sadv-diff-coefficients} are
available, e.g., in \cite{Roz-B90} and under weaker assumptions see, e.g., 
\cite{MikRoz98,LotRoz06a}.

We will continue to use the notation from the previous section: $h$ is a
step of uniform discretization of the interval $[0,T],$ $N=T/h$ and $%
t_{k}=kh,$ $k=0,\ldots ,N.$ We apply the trapezoidal rule in time to the
SPDE \eqref{eq:sadv-diff}: 
\begin{gather}
u^{k+1}(x)=u^{k}(x)+h[\tilde{\mathcal{L}}u^{k+1/2}(x)-\frac{1}{2}%
\sum_{l=1}^{r}\mathcal{M}_{l}g_{l}(x)+f(x)]  \label{eq:sadv-diff-cn} \\
+\sum_{l=1}^{r}\left[ \mathcal{M}_{l}u^{k+1/2}(x)+g_{l}(x)\right] \sqrt{h}%
\left( \xi _{lh}\right) _{k+1},\;x\in \mathcal{D},  \notag \\
u^{0}(x)=u_{0}(x),  \notag
\end{gather}%
where $u^{k}(x)$ approximates $u(t_{k},x)$, $u^{k+1/2}=(u^{k+1}+u^{k})/2,$
and $\left( \xi _{lh}\right) _{k}$ are i.i.d. random variables so that 
\begin{equation}
\xi _{h}=\left\{ 
\begin{array}{c}
\xi ,\;|\xi |\leq A_{h}, \\ 
A_{h},\;\xi >A_{h}, \\ 
-A_{h},\;\xi <-A_{h},%
\end{array}%
\right.  \label{Fin61}
\end{equation}%
with $\xi $ $\sim $ $\mathcal{N}(0,1)$ and $A_{h}=\sqrt{2p|\ln h|}$ with $%
p\geq 1.$ We note that the cut-off of the Gaussian random variables is
needed in order to ensure that the implicitness of \eqref{eq:sadv-diff-cn}
does not lead to non-existence of the second moment of $u^{k}(x)$ \cite%
{MilRT02,MilTre-B04}. Based on the standard results of numerics for SDE \cite%
{MilTre-B04}, it is natural to expect that under some regularity assumptions
on the coefficients and the initial condition of \eqref{eq:sadv-diff}, the
approximation $u^{k}(x)$ from (\ref{eq:sadv-diff-cn}) converges with order $%
1/2$ in the mean-square sense and with order $1$ in the weak sense and in
the latter case one can use discrete random variables $\zeta _{l,k+1}$ from (%
\ref{eq:weak-euler-discrete-apprx-gauss}) instead of $\left( \xi
_{lh}\right) _{k+1}$ (see also e.g. {\color{black}\cite{Deb11,GreKlo96,KloSho01}} but we are not
proving such a result here). 

In what follows it will be convenient to also use the notation: $%
u_{H}^{k}(x;\phi (\cdot ))=u_{H}^{k}(x;\phi (\cdot );\left( \xi _{lh}\right)
_{k},l=1,\ldots ,r)$ for the approximation \eqref{eq:sadv-diff-cn} of the
solution $u(t_{k},x)$ to the SPDE \eqref{eq:sadv-diff} with $f(x)=0$ and $%
g_{l}(x)=0$ for all $l$ (homogeneous SPDE) and with the initial condition $%
\phi (\cdot )$ prescribed at time $t=t_{k-1};$ $u_{O}^{k}(x)=u_{O}^{k}(x;%
\left( \xi _{lh}\right) _{k},l=1,\ldots ,r)$ for the approximation %
\eqref{eq:sadv-diff-cn} of the solution $u(t_{k},x)$ to the SPDE %
\eqref{eq:sadv-diff} with the initial condition $\phi (x)=0$ prescribed at
time $t=t_{k-1}.$ Note that $u_{O}^{k}(x)=0$ if $f(x)=0$ and $g_{l}(x)=0$
for all $l.$

Let $\left\{ e_{i}\right\} =\left\{ e_{i}(x)\right\} _{i\geq 1}$ be a
complete orthonormal system (CONS) in $L^{2}(\mathcal{D})$ with boundary
conditions satisfied and $(\cdot ,\cdot )$ be the inner product in that
space. Then we can write 
\begin{equation}
u^{k-1}(x)=\sum_{i=1}^{\infty }c_{i}^{k-1}e_{i}(x)  \label{agan1}
\end{equation}%
with $c_{i}^{k-1}=(u^{k-1},e_{i})$ and, due to the SPDE's linearity: 
\begin{equation*}
u^{k}(x)=u_{O}^{k}(x)+\sum_{i=1}^{\infty }c_{i}^{k-1}u_{H}^{k}(x;e_{i}(\cdot
)).
\end{equation*}%
We have 
\begin{equation*}
c_{l}^{0}=(u_{0},e_{l}),\ \ \ c_{l}^{k}=q_{Ol}^{k}+\sum_{i=1}^{\infty
}c_{i}^{k-1}q_{Hli}^{k},\ \ l=1,2,\ldots ,\ \ k=1,\ldots ,N,
\end{equation*}%
where $q_{Ol}^{k}=(u_{O}^{k},e_{l})$ and $q_{Hli}^{k}=(u_{H}^{k}(\cdot
;e_{i}),e_{l}(\cdot )).$

Using (\ref{agan1}), we represent the second moment of the approximation $%
u^{k}(x)$ from \eqref{eq:sadv-diff-cn} of the solution $u(t_{k},x)$ to the
SPDE \eqref{eq:sadv-diff} as follows%
\begin{equation}
\mathbb{E}[u^{k}(x)]^{2}=\sum_{i,j=1}^{\infty }C_{ij}^{k}e_{i}(x)e_{j}(x),
\label{Eu2}
\end{equation}%
where the covariance matrix $C_{ij}^{k}=\mathbb{E}[c_{i}^{k}c_{j}^{k}].$
Introducing also the means $M_{i}^{k},$ one can obtain the recurrent
relations in $k:$ 
\begin{eqnarray}
M_{i}^{0} &=&c_{i}^{0}=(u_{0},e_{i}),\ \ C_{ij}^{0}=c_{i}^{0}c_{j}^{0},
\label{agan2} \\
M_{i}^{k} &=&\mathbb{E}[q_{Oi}^{k}]+\sum_{l=1}^{\infty }M_{l}^{k-1}\mathbb{E}%
[q_{Hil}^{k}],  \notag \\
C_{ij}^{k} &=&\mathbb{E}[q_{Oi}^{k}q_{Oj}^{k}]+\sum_{l=1}^{\infty
}M_{l}^{k-1}\left( \mathbb{E}[q_{Oi}^{k}q_{Hjl}^{k}]+\mathbb{E}%
[q_{Oj}^{k}q_{Hil}^{k}]\right) +\sum_{l,p=1}^{\infty }C_{lp}^{k-1}\mathbb{E}%
[q_{Hil}^{k}q_{Hjp}^{k}],  \notag \\
i,j &=&1,2,\ldots ,\ \ k=1,\ldots ,N.  \notag
\end{eqnarray}%
Since the coefficients of the SPDE \eqref{eq:sadv-diff} are time
independent, all the expectations involving the quantities $q_{Oi}^{k}$ and $%
q_{Hil}^{k}$ in (\ref{agan2}) do not depend on $k$ and hence it is
sufficient to compute them just once, on a single step $k=1,$ and we get 
\begin{eqnarray}
M_{i}^{0} &=&c_{i}^{0}=(u_{0},e_{i}),\ \ C_{ij}^{0}=c_{i}^{0}c_{j}^{0},
\label{recMC} \\
M_{i}^{k} &=&\mathbb{E}[q_{Oi}^{1}]+\sum_{l=1}^{\infty }M_{l}^{k-1}\mathbb{E}%
[q_{Hil}^{1}],  \notag \\
C_{ij}^{k} &=&\mathbb{E}[q_{Oi}^{1}q_{Oj}^{1}]+\sum_{l=1}^{\infty
}M_{l}^{k-1}\left( \mathbb{E}[q_{Oi}^{1}q_{Hjl}^{1}]+\mathbb{E}%
[q_{Oj}^{1}q_{Hil}^{1}]\right) +\sum_{l,p=1}^{\infty }C_{lp}^{k-1}\mathbb{E}%
[q_{Hil}^{1}q_{Hjp}^{1}],  \notag \\
i,j &=&1,2,\ldots ,\ \ k=1,\ldots ,N.  \notag
\end{eqnarray}%
These expectations can be approximated by quadrature rules from Section~2.1.
If the number of noises $r$ is small, then it is natural to use the tensor
product rule (\ref{appi}) with one-dimensional Gauss--Hermite quadratures of
order $n=2$ or $3$ (note that when $r=1,$ we can use just a one-dimensional
Gauss--Hermite quadrature of order $n=2$ or $3).$ If the number of noises $r$
is large then it might be beneficial to use the sparse grid quadrature (\ref%
{eq:smolyak-tensor-like}) of level $L=2$ or $3.$ More specifically, 
\begin{eqnarray}
\mathbb{E}[q_{Oi}^{1}] &\doteq &\sum_{p=1}^{\eta }(u_{O}^{1}(\cdot ;\mathrm{y%
}_{p}),e_{i}(\cdot ))\mathsf{W}_{p},\ \ \mathbb{E}[q_{Hil}^{1}]\doteq
\sum_{p=1}^{\eta }(u_{H}^{1}(\cdot ;e_{l};\mathrm{y}_{p}),e_{i}(\cdot ))%
\mathsf{W}_{p},  \label{app_expec} \\
\mathbb{E}[q_{Oi}^{1}q_{Oj}^{1}] &\doteq &\sum_{p=1}^{\eta }(u_{O}^{1}(\cdot
;\mathrm{y}_{p}),e_{i}(\cdot ))(u_{O}^{1}(\cdot ;\mathrm{y}_{p}),e_{j}(\cdot
))\mathsf{W}_{p},  \notag \\
\mathbb{E}[q_{Oi}^{1}q_{Hjl}^{1}] &\doteq &\sum_{p=1}^{\eta
}(u_{O}^{1}(\cdot ;\mathrm{y}_{p}),e_{i}(\cdot ))(u_{H}^{1}(\cdot ;e_{l};%
\mathrm{y}_{p}),e_{j}(\cdot ))\mathsf{W}_{p},  \notag \\
\mathbb{E}[q_{Hil}^{1}q_{Hjk}^{1}] &\doteq &\sum_{p=1}^{\eta
}(u_{H}^{1}(\cdot ;e_{l};\mathrm{y}_{p}),e_{i}(\cdot ))(u_{H}^{1}(\cdot
;e_{k};\mathrm{y}_{p}),e_{j}(\cdot ))\mathsf{W}_{p},  \notag
\end{eqnarray}%
where $\mathrm{y}_{p}\in \mathbb{R}^{r}$ are nodes of the quadrature, $%
\mathsf{W}_{p}$ are the corresponding quadrature weights, and $\eta =n^{r}$
in the case of the tensor product rule (\ref{appi}) with one-dimensional
Gauss--Hermite quadratures of order $n$ or $\eta $ is the total number of
nodes $\#S$ used by the sparse-grid quadrature (\ref{eq:smolyak-tensor-like}%
) of level $L.$ To find $u_{O}^{1}(x;\mathrm{y}_{p})$ and $u_{H}^{1}(x;e_{l};%
\mathrm{y}_{p}),$ we need to solve the corresponding elliptic PDE problems,
which we do using the spectral method in physical space, i.e., using a
truncation of the CONS $\left\{ e_{l}\right\} _{l=1}^{l_{\ast }}$ to
represent the numerical solution.

To summarize, we formulate the following deterministic recursive algorithm
for the second-order moments of the solution to the SPDE problem (\ref%
{eq:sadv-diff}).

\begin{algo}
\label{algo:sadv-diff-s4-scm-mom} Choose the algorithm's parameters: a
complete orthonormal basis $\{e_{l}(x)\}_{l\geq 1}$ in $L^{2}(\mathcal{D})$
and its truncation $\{e_{l}(x)\}_{l=1}^{l_{\ast }}$; a time step size $h$;
and a quadrature rule (i.e., nodes $\mathrm{y}_{p}\ $and the quadrature
weights $\mathsf{W}_{p},$ $p=1,\ldots ,\eta )$.

Step 1. For each $p=1,\ldots ,\eta $ and $l=1,\ldots ,l_{\ast },$ find
approximations $\bar{u}_{O}^{1}(x;\mathrm{y}_{p})\approx u_{O}^{1}(x;\mathrm{%
y}_{p})$ and $\bar{u}_{H}^{1}(x;e_{l};\mathrm{y}_{p})\approx
u_{H}^{1}(x;e_{l};\mathrm{y}_{p})$ using the spectral method in physical
space.

Step 2. Using the quadrature rule, approximately find the expectations as in 
$(\ref{app_expec})$ but with the approximate $\bar{u}_{O}^{1}(x;\mathrm{y}%
_{p})$ and $\bar{u}_{H}^{1}(x;e_{l};\mathrm{y}_{p})$ instead of $u_{O}^{1}(x;%
\mathrm{y}_{p})$ and $u_{H}^{1}(x;e_{l};\mathrm{y}_{p}),$ respectively.

Step 3. Recursively compute the approximations of the means $M_{i}^{k},$ $%
i=1,\ldots ,l_{\ast },$ and covariance matrices $\{C_{ij}^{k},$ $%
i,j=1,\ldots ,l_{\ast }\}$ for $k=1,\ldots ,N$ according to $(\ref{recMC})$
with the approximate expectations found in Step 2 instead of the exact ones.

Step 4. Compute the approximation of the second-order moment $\mathbb{E}%
\lbrack u^{k}(x)\rbrack^{2}$ using $(\ref{Eu2})$ with the approximate
covariance matrix found in Step~3 instead of the exact one $\{C_{ij}^{k}\}.$
\end{algo}

We emphasize that Algorithm~\ref{algo:sadv-diff-s4-scm-mom} for computing
moments does not have a statistical error. {\color{black}Based on the error estimate  in Proposition \ref{prop:weak-apprx-euler-sg}, we expect the one-step error of SGC  for our recursive algorithm 
  is of order $h^L$. Hence, we expect the total global error from trapezoidal rule in time and SGC      to be $O(h)+O(h^{L-1})$.} Error analysis of this algorithm
will be considered elsewhere. 

\begin{rem}
Algorithms analogous to Algorithm~\ref{algo:sadv-diff-s4-scm-mom} can also
be constructed based on other time-discretizations methods than the
trapezoidal rule used here or based on other types of SPDE approximations,
e.g. one can exploit the Wong-Zakai approximation.
\end{rem}

\begin{rem}
\label{rm:s4-scm-cost} The cost of this algorithm is, similar to the
algorithm in \cite{ZhangRTK12}, $\frac{T}{\Delta }\eta l_{\ast}^{4}$ and the
storage is $\eta l_{\ast}^{2}$. The total cost can be reduced by employing
some reduced order methods in physical space and be proportional to $%
l_{\ast}^{2}$ instead of $l_{\ast}^{4}$. The discussion on computational
efficiency of the recursive Wiener chaos method is also valid here, see \cite%
[Remark 4.1]{ZhangRTK12}.
\end{rem}

{\color{black}
\begin{rem} 
Choosing an orthonormal basis  is an important topic in the research of spectral methods, 
which can be found in \cite{GotOrs-B77} and many subsequent works. Here we choose Fourier basis for  Problem  \ref{eq:sadv-diff} because of   periodic boundary conditions.
\end{rem}
}
\section{Numerical experiments\label{sec:num-experiments-sg}}

In this section we illustrate via three examples how the SGC algorithms can
be used for the weak-sense approximation of SDE and SPDE. The first example
is a scalar SDE with multiplicative noise, where we show that the SGC
algorithm's error is small when the noise magnitude is small. We also
observe that when the noise magnitude is large, the SGC algorithm does not
work well. In the second example we demonstrate that the SGC can be
successfully used for simulating Burgers equation with additive noise when
the integration time is relatively small. In the last example we show that
the recursive algorithm from Section~\ref{sec:recursive-sg-advdiff} works
effectively for computing moments of the solution to an advection-diffusion
equation with multiplicative noise over a longer integration time.

In all the tests we limit the dimension of random spaces by $40$, which is
an empirical limitation of the SGC of Smolyak on the dimensionality \cite%
{Pet03}. Also, we take the sparse grid level less than or equal to five in
order to avoid an excessive number of sparse grid points. All the tests were
run using Matlab R2012b on a Macintosh desktop computer with Intel Xeon CPU
E5462 (quad-core, 2.80 GHz).

\begin{exm}[modified Cox-Ingersoll-Ross (mCIR), see e.g. {\protect\cite%
{ComGR07}}]
\label{exm:mcir} 
\upshape%
Consider the Ito SDE%
\begin{equation}
dX=-\theta _{1}X\,dt+\theta _{2}\sqrt{1+X^{2}}\,dw(t),\quad X(0)=x_{0}.
\label{eq:m-cir}
\end{equation}%
For $\theta _{2}^{2}-2\theta _{1}\neq 0$, the first two moments of $X(t)$
are equal to%
\begin{equation*}
\mathbb{E}X(t)=x_{0}\exp (-\theta _{1}t),\quad \mathbb{E}X^{2}(t)=-\frac{%
\theta _{2}^{2}}{\theta _{2}^{2}-2\theta _{1}}+(x_{0}^{2}+\frac{\theta
_{2}^{2}}{\theta _{2}^{2}-2\theta _{1}})\exp ((\theta _{2}^{2}-2\theta
_{1})t).
\end{equation*}%
In this example we test the SGC algorithms based on the Euler scheme (\ref%
{eq:ito-sde-vector-euler}) and on the second-order weak scheme %
\eqref{eq:ito-sde-vec-weak-2nd-order}. We compute the first two moments of
the SDE's solution and {\color{black}use the relative errors to  measure  accuracy  of} the algorithms as 
\begin{equation}
\rho _{1}^{r}(T)=\frac{\left\vert \mathbb{E}X(T)-\mathbb{E}X_{N}\right\vert 
}{\left\vert \mathbb{E}X(T)\right\vert },\quad \rho _{2}^{r}(T)=\frac{%
\left\vert \mathbb{E}X^{2}(T)-\mathbb{E}X_{N}^{2}\right\vert }{\mathbb{E}%
X^{2}(T)}.  \label{eq:error-measure-sode-firstTwomom}
\end{equation}

Table \ref{tbl:mcir-1} presents the errors for the SGC algorithms based on
the Euler scheme (left) and on the second-order scheme %
\eqref{eq:ito-sde-vec-weak-2nd-order} (right), when the noise magnitude is
small. For the parameters given in the table's description, the exact values
(up to 4 d.p.) of the first and second moments are $3.679\times 10^{-2}$ and 
$4.162\times 10^{-2}$, respectively. We see that increase of the SGC level $%
L $ above $2$ in the Euler scheme case and above $3$ in the case of the
second-order scheme does not improve accuracy. When the SGC error is
relatively small in comparison with the error due to time discretization, we
observe decrease of the overall error of the algorithms in $h$: proportional
to $h$ for the Euler scheme and to $h^{2}$ for the second-order scheme. We
underline that in this experiment the noise magnitude is small.
\end{exm}

\begin{table}[tbph]
\caption{Comparison of the SGC algorithms based on the Euler scheme (left)
and on the second-order scheme \eqref{eq:ito-sde-vec-weak-2nd-order}
(right). The parameters of the model (\protect\ref{eq:m-cir}) are $x_{0}=0.1$%
, $\protect\theta _{1}=1,$ $\protect\theta _{2}=0.3,$ and $T=1$. }
\label{tbl:mcir-1}
\begin{center}
\scalebox{0.75}{
\begin{tabular}{cccccc|ccccc}
\hline
$h$ & $L$ & $\rho _{1}^{r}(1)$ & order & $\rho _{2}^{r}(1)$ & 
\multicolumn{1}{c|}{order} & $L$ & $\rho _{1}^{r}(1)$ & order & $\rho
_{2}^{r}(1)$ & order \\ \hline\hline
\multicolumn{1}{l}{5$\times 10^{-1}$} & 2 & 3.20$\times 10^{-1}$ & -- & 3.72$\times 10^{-1}$ & \multicolumn{1}{c|}{--} & 3 &  $\mathbf{6.05\times 10^{-2}}$ & -- & 
8.52$\times 10^{-2}$ & -- \\ \hline
\multicolumn{1}{l}{2.5$\times 10^{-1}$} & 2 & 1.40$\times 10^{-1}$ & 1.2 &  1.40$\times 10^{-1}$ & \multicolumn{1}{c|}{1.4} & 3 & 1.14$\times 10^{-2}$ & 
2.4 & 2.10$\times 10^{-2}$ & 2.0 \\ \hline
\multicolumn{1}{l}{1.25$\times 10^{-1}$} & 2 & $\mathbf{6.60\times 10^{-2}}$ & 1.1 &  4.87$\times 10^{-2}$ & \multicolumn{1}{c|}{1.5} & 3 & 1.75$\times 10^{-3}$ & 
2.7 & 6.73$\times 10^{-3}$ & 1.6 \\ \hline
\multicolumn{1}{l}{6.25$\times 10^{-2}$} & 2 & 3.21$\times 10^{-2}$ & 1.0 &  8.08$\times 10^{-3}$ & \multicolumn{1}{c|}{2.6} & 4 & 3.64$\times 10^{-4}$ & 
2.3 & 1.21$\times 10^{-3}$ & 2.5 \\ \hline
\multicolumn{1}{l}{3.125$\times 10^{-2}$} & 2 & 1.58$\times 10^{-2}$ & 1.0 &  1.12$\times 10^{-2}$ & \multicolumn{1}{c|}{-0.5} & 4 & 8.48$\times 10^{-4}$
& -1.2 & 3.75$\times 10^{-4}$ & 1.7 \\ \hline\hline
\multicolumn{1}{l}{2.5$\times 10^{-2}$} & 2 & 1.26$\times 10^{-2}$ &  & 1.49$\times 10^{-2}$ & \multicolumn{1}{c|}{} & 2 & 9.02$\times 10^{-4}$ &  & 5.72$\times 10^{-2}$ &  \\ \hline
\multicolumn{1}{l}{2.5$\times 10^{-2}$} & 3 & 1.26$\times 10^{-2}$ &  & 1.48$\times 10^{-2}$ & \multicolumn{1}{c|}{} & 3 & 9.15$\times 10^{-5}$ &  & 2.84$\times 10^{-3}$ &  \\ \hline
\multicolumn{1}{l}{2.5$\times 10^{-2}$} & 4 & 1.26$\times 10^{-2}$ &  & 1.55$\times 10^{-2}$ & \multicolumn{1}{c|}{} & 4 & 1.06$\times 10^{-4}$ &  & 2.77$\times 10^{-4}$ &  \\ \hline
\multicolumn{1}{l}{2.5$\times 10^{-2}$} & 5 & 1.26$\times 10^{-2}$ &  & 1.56$\times 10^{-2}$ & \multicolumn{1}{c|}{} & 5 & 1.06$\times 10^{-4}$ &  & 1.81$\times 10^{-4}$ &  \\ \hline\hline
\end{tabular}}
\end{center}
\end{table}

In Table \ref{tbl:mcir-21} we give results of the numerical experiment when
the noise magnitude is not small. For the parameters given in the table's
description, the exact values (up to 4 d.p.) of the first and second moments
are $0.2718$ and $272.3202$, respectively. Though for the Euler scheme there
is a proportional to $h$ decrease of the error in computing the mean, there
is almost no decrease of the error in the rest of this experiment. The large
value of the second moment apparently affects efficiency of the SGC here.
For the Euler scheme, increasing $L$ and decreasing $h$ can slightly improve
accuracy in computing the second moment, e.g. the smallest relative error
for the second moment is $56.88\%$ when $h=0.03125$ and $L=5$ (this level
requires 750337 sparse grid points) out of the considered cases of $h=0.5,$\ 
$0.25,$\ $0.125,$\ $0.0625,\ $and $0.03125$ and $L\leq 5$. For the mean,
increase of the level $L$ from $2$ to $3,$ $4$ or $5$ does not improve
accuracy. For the second-order scheme \eqref{eq:ito-sde-vec-weak-2nd-order},
relative errors for the mean can be decreased by increasing $L$ for a fixed $%
h$: e.g., for $h=0.25$, the relative errors are $0.5121$ $0.1753$, $0.0316$
and $0.0086$ when $L=2,$\ $3,$\ $4,\ $and $5,$ respectively.

We also see in Table~\ref{tbl:mcir-21} that the SGC algorithm based on the
second-order scheme may not admit higher accuracy than the one based on the
Euler scheme, e.g. for $h=0.5,~0,25,~0.125$ the second-order scheme yields
higher accuracy while the Euler scheme demonstrates higher accuracy for
smaller $h=0.0625$ and $0.03125$. Further decrease in $h$ was not considered
because this would lead to increase of the dimension of the random space
beyond 40 when the sparse grid of Smolyak \eqref{eq:smolyak-tensor-like} may
fail and the SGC algorithm may also lose its competitive edge with Monte
Carlo-type techniques.

\begin{table}[tbph]
\caption{Comparison of the SGC algorithms based on the Euler scheme (left)
and on the second-order scheme \eqref{eq:ito-sde-vec-weak-2nd-order}
(right). The parameters of the model (\protect\ref{eq:m-cir}) are $%
x_{0}=0.08 $, $\protect\theta _{1}=-1,$ $\protect\theta _{2}=2,$ and $T=1$.
The sparse grid level $L=4$.}
\label{tbl:mcir-21}
\begin{center}
\begin{tabular}{cccc|cc}
\hline
$h$ & $\rho _{1}^{r}(1)$ & order & $\rho _{2}^{r}(1)$ & $\rho _{1}^{r}(1)$ & 
$\rho _{2}^{r}(1)$ \\ \hline\hline
\multicolumn{1}{l}{5$\times 10^{-1}$} & 1.72$\times 10^{-1}$ & -- & 9.61$%
\times 10^{-1}$ & 2.86$\times 10^{-2}$ & 7.69$\times 10^{-1}$ \\ \hline
\multicolumn{1}{l}{2.5$\times 10^{-1}$} & 1.02$\times 10^{-1}$ & 0.8 & 8.99$%
\times 10^{-1}$ & 8.62$\times 10^{-3}$ & 6.04$\times 10^{-1}$ \\ \hline
\multicolumn{1}{l}{1.25$\times 10^{-1}$} & 5.61$\times 10^{-2}$ & 0.9 & 7.87$%
\times 10^{-1}$ & 1.83$\times 10^{-2}$ & 7.30$\times 10^{-1}$ \\ \hline
\multicolumn{1}{l}{6.25$\times 10^{-2}$} & 2.96$\times 10^{-2}$ & 0.9 & 6.62$%
\times 10^{-1}$ & 3.26$\times 10^{-2}$ & 8.06$\times 10^{-1}$ \\ \hline
\multicolumn{1}{l}{3.125$\times 10^{-2}$} & 1.52$\times 10^{-2}$ & 1.0 & 5.64%
$\times 10^{-1}$ & 4.20$\times 10^{-2}$ & 8.40$\times 10^{-1}$ \\ \hline\hline
\end{tabular}%
\end{center}
\end{table}

Via this example we have shown that the SGC algorithms based on first- and
second-order schemes can produce sufficiently accurate results when noise
magnitude is small and that the second-order scheme is preferable since for
the same accuracy it uses random spaces of lower dimension than the
first-order Euler scheme, compare e.g. the error values highlighted by bold
font in Table~\ref{tbl:mcir-1} and see also the discussion at the end of
Section~2.2. When the noise magnitude is large (see Table~\ref{tbl:mcir-21}%
), the SGC algorithms do not work well as it was predicted in Section~2.3.

\begin{exm}[Burgers equation with additive noise]
\upshape%
Consider the stochastic Burgers equation \cite{DaPDT94,HouLRZ06}:%
\begin{equation}
du+u\frac{\partial u}{\partial x}dt=\nu \frac{\partial ^{2}u}{\partial x^{2}}%
dt+\sigma \cos (x)d{w},\quad 0\leq x\leq \ell ,\quad \nu >0
\label{eq:burgers-additive2}
\end{equation}%
with the initial condition $u_{0}(x)=2\nu \frac{2\pi }{\ell }\frac{\sin (%
\frac{2\pi }{\ell }x)}{a+\cos (\frac{2\pi }{\ell }x)},$ $a>1$, and periodic
boundary conditions. In the numerical tests the used values of the
parameters are $\ell =2\pi $ and $a=2$.
\end{exm}

Apply the Fourier collocation method in physical space and the trapezoidal
rule in time to (\ref{eq:burgers-additive2}): 
\begin{equation}
\frac{\vec{u}_{j+1}-\vec{u}_{j}}{h}-\nu D^{2}\frac{\vec{u}_{j+1}+\vec{u}_{j}%
}{2}=-\frac{1}{2}D(\frac{\vec{u}_{j+1}+\vec{u}_{j}}{2})^{2}+\sigma \Gamma 
\sqrt{h}\xi _{j},  \label{eq:cn-burgers-additive}
\end{equation}%
where $\vec{u}_{j}=(u(t_{j},x_{1}),\ldots ,u(t_{j},x_{M}))^{\intercal
},\quad t_{j}=jh,$ $D$ is the Fourier spectral differential matrix, $\xi
_{j} $ are i.i.d $\mathcal{N}(0,1)$ random variables, and $\Gamma =(\cos
(x_{1}),\ldots ,\cos (x_{M}))^{\intercal }.$ The Fourier collocation points
are $x_{m}=m\frac{\ell }{M}$ ($m=1,\ldots ,M$) in physical space and in the
experiment we used $M=100$. We aim at computing moments of $\vec{u}_{j},$
which are integrals with respect to the Gaussian measure corresponding to
the collection of $\xi _{j},$ and we approximate these integrals using the
SGC from Section~2. The use of the SGC amounts to substituting $\xi _{j}$ in %
\eqref{eq:cn-burgers-additive} by sparse-grid nodes, which results in a
system of (deterministic) nonlinear equations of the form %
\eqref{eq:cn-burgers-additive}. To solve the nonlinear equations, we used
the fixed-point iteration method with tolerance $h^{2}/100$.

The errors in computing the first and second moments are measured as follows 
\begin{eqnarray}
\rho _{1}^{r,2}(T) &=&\frac{\left\Vert \mathbb{E}u_{\mathrm{ref}}(T,\cdot )-%
\mathbb{E}u_{\mathrm{num}}(T,\cdot )\right\Vert }{\left\Vert \mathbb{E}u_{%
\mathrm{ref}}(T,\cdot )\right\Vert },\quad \rho _{2}^{r,2}(T)=\frac{%
\left\Vert \mathbb{E}u_{\mathrm{ref}}^{2}(T,\cdot )-\mathbb{E}u_{\mathrm{num}%
}^{2}(T,\cdot )\right\Vert }{\left\Vert \mathbb{E}u_{\mathrm{ref}%
}^{2}(T,\cdot )\right\Vert },\qquad  \label{eq:error-measure-l2} \\
\rho _{1}^{r,\infty }(T) &=&\frac{\left\Vert \mathbb{E}u_{\mathrm{ref}%
}(T,\cdot )-\mathbb{E}u_{\mathrm{num}}(T,\cdot )\right\Vert _{\infty }}{%
\left\Vert \mathbb{E}u_{\mathrm{ref}}(T,\cdot )\right\Vert _{\infty }},\ \
\rho _{2}^{r,\infty }(T)=\frac{\left\Vert \mathbb{E}u_{\mathrm{ref}%
}^{2}(T,\cdot )-\mathbb{E}u_{\mathrm{num}}^{2}(T,\cdot )\right\Vert _{\infty
}}{\left\Vert \mathbb{E}u_{\mathrm{ref}}^{2}(T,\cdot )\right\Vert _{\infty }}%
,  \notag
\end{eqnarray}%
where $\left\Vert v(\cdot )\right\Vert =\displaystyle\left( \frac{2\pi }{M}%
\sum_{m=1}^{M}v^{2}(x_{m})\right) ^{1/2}$, $\left\Vert v(\cdot )\right\Vert
_{\infty }=\displaystyle\max_{1\leq m\leq M}\left\vert v(x_{m})\right\vert $%
, $x_{m}$ are the Fourier collocation points, and $u_{\mathrm{num}}$ and $%
~u_{\mathrm{ref}}$ are the numerical solution obtained by the SGC algorithm
and the reference solution, respectively. The first and second moments of
the reference solution $u_{\mathrm{ref}}$ were computed by the same solver
in space and time \eqref{eq:cn-burgers-additive} but accompanied by the
Monte Carlo method with a large number of realizations ensuring that the
statistical errors were negligible.

First, we choose $\nu =0.1$ and $\sigma =1$. We obtain the reference
solution with $h=10^{-4}$ and $1.92\times 10^{6}$ Monte Carlo realizations.
The corresponding statistical error is $1.004\times 10^{-3}$ for the mean
(maximum of the statistical error for $\mathbb{E}u_{\mathrm{ref}}(0.5,x_{j})$%
) and $9.49\times 10^{-4}$ for the second moment (maximum of the statistical
error for $\mathbb{E}u_{\mathrm{ref}}^{2}(0.5,x_{j})$) with $95\%$
confidence interval, and the corresponding estimates of $L^{2}$-norm of the
moments are $\left\Vert \mathbb{E}u_{\mathrm{ref}}(0.5,\cdot )\right\Vert
\doteq 0.18653$ and $\left\Vert \mathbb{E}u_{\mathrm{ref}}^{2}(0.5,\cdot
)\right\Vert \doteq 0.72817$. We see from the results of the experiment
presented in Table~\ref{tbl:sc-sburgers-CN-large-para} that for $L=2$ the
error in computing the mean decreases when $h$ decreases up to $h=0.05$ but
the accuracy does not improve with further decrease of $h$. For the second
moment, we observe no improvement in accuracy with decrease of $h$. For $%
L=4, $ the error in computing the second moment decreases with $h$. When $%
h=0.0125 $, increasing the sparse grid level improves the accuracy for the
mean: $L=3$ yields $\rho _{1}^{r,2}(0.5)\doteq 9.45\times 10^{-3}$ and $L=4$
yields $\rho _{1}^{r,2}(0.5)\doteq 8.34\times 10^{-3}$. As seen in Table~\ref%
{tbl:sc-sburgers-CN-large-para}, increase of the level $L$ also improves
accuracy for the second moment when $h=0.05,$ $0.25,$ and $0.125$.

\begin{table}[tbph]
\caption{Errors of the SGC algorithm to the stochastic Burgers equation 
\eqref{eq:burgers-additive2} with parameters $T=0.5$, $\protect\nu =0.1$ and 
$\protect\sigma =1$. }
\label{tbl:sc-sburgers-CN-large-para}
\begin{center}
\scalebox{0.79}{
\begin{tabular}{ccc|ccc}
\hline
$h$ & $\rho _{1}^{r,2}(0.5),\ L=2$ & $\rho _{1}^{r,2}(0.5),$ $L=3$ & $\rho
_{2}^{r,2}(0.5),\ L=2$ & $\rho _{2}^{r,2}(0.5),$ $L=3$ & $\rho
_{2}^{r,2}(0.5),$ $L=4$ \\ \hline\hline
\multicolumn{1}{l}{2.5$\times 10^{-1}$} & 1.28$\times 10^{-1}$ & 1.3661$  \times 10^{-1}$ & 4.01$\times 10^{-2}$ & 1.05$\times 10^{-2}$ & 1.25$\times
10^{-2}$ \\ \hline
\multicolumn{1}{l}{1.00$\times 10^{-1}$} & 4.70$\times 10^{-2}$ & 5.3874$  \times 10^{-2}$ & 4.48$\times 10^{-2}$ & 4.82$\times 10^{-3}$ & 4.69$\times
10^{-3}$ \\ \hline
\multicolumn{1}{l}{5.00$\times 10^{-2}$} & 2.75$\times 10^{-2}$ & 2.7273$  \times 10^{-2}$ & 4.73$\times 10^{-2}$ & 5.89$\times 10^{-3}$ & 2.82$\times
10^{-3}$ \\ \hline
\multicolumn{1}{l}{2.50$\times 10^{-2}$} & 2.51$\times 10^{-2}$ & 1.4751$  \times 10^{-2}$ & 4.87$\times 10^{-2}$ & 6.92$\times 10^{-3}$ & 2.34$\times
10^{-3}$ \\ \hline
\multicolumn{1}{l}{1.25$\times 10^{-2}$} & 2.67$\times 10^{-2}$ & 9.4528$  \times 10^{-3}$ & 4.95$\times 10^{-2}$ & 7.51$\times 10^{-3}$ & 2.29$\times
10^{-3}$ \\ \hline\hline
\end{tabular}
}
\end{center}
\end{table}

Second, we choose $\nu =1$ and $\sigma =0.5$. We obtain the first two
moments of the reference $u_{\mathrm{ref}}$ using $h=10^{-4}$ and the Monte
Carlo method with $3.84\times 10^{6}$ realizations. The corresponding
statistical error is $3.2578\times 10^{-4}$ for the mean and $2.2871\times
10^{-4}$ for the second moment with $95\%$ confidence interval, and the
corresponding estimates of $L^{2}$-norm of the moments are $\left\Vert 
\mathbb{E}u_{\mathrm{ref}}(0.5,\cdot )\right\Vert \doteq 1.11198$ and $%
\left\Vert \mathbb{E}u_{\mathrm{ref}}^{2}(0.5,\cdot )\right\Vert \doteq
0.66199$.

The results of the experiment are presented in Table~\ref{tbl:sc-sburgers-CN}%
. We see that accuracy is sufficiently high and there is some decrease of
errors with decrease of time step $h.$ However, as expected, no convergence
in $h$ is observed and further numerical tests (not presented here) showed
that taking $h$ smaller than $1.25\times 10^{-2}$ and level $L=2$ or $3$
does not improve accuracy. In additional experiments we also noticed that
there was no improvement of accuracy for the mean when we increased the
level $L$ up to $5$. For the second moment, we observe some improvement in
accuracy when $L$ increases from 2 to 3 (see Table~\ref{tbl:sc-sburgers-CN})
but additional experiments (not presented here) showed that further increase
of $L$ (up to $5)$ does not reduce the errors.

\begin{table}[tbph]
\caption{Errors of the SGC algorithm applied to the stochastic Burgers
equation \eqref{eq:burgers-additive2} with parameters $\protect\nu =1$, $%
\protect\sigma =0.5,$ and $T=0.5$.}
\label{tbl:sc-sburgers-CN}\centering%
\begin{tabular}{ccc|c}
\hline
$h$ & $\rho _{1}^{r,2}(0.5),\ L=2$ & $\rho _{2}^{r,2}(0.5),\ L=2$ & $\rho
_{2}^{r,2}(0.5),\ L=3$ \\ \hline\hline
\multicolumn{1}{l}{2.5$\times 10^{-1}$} & 4.94$\times 10^{-3}$ & 8.75$\times
10^{-3}$ & 8.48$\times 10^{-3}$ \\ \hline
\multicolumn{1}{l}{1$\times 10^{-1}$} & 8.20$\times 10^{-4}$ & 1.65$\times
10^{-3}$ & 1.13$\times 10^{-3}$ \\ \hline
\multicolumn{1}{l}{5$\times 10^{-2}$} & 4.88$\times 10^{-4}$ & 1.18$\times
10^{-3}$ & 6.47$\times 10^{-4}$ \\ \hline
\multicolumn{1}{l}{2.5$\times 10^{-2}$} & 3.83$\times 10^{-4}$ & 1.08$\times
10^{-3}$ & 5.01$\times 10^{-4}$ \\ \hline
\multicolumn{1}{l}{1.25$\times 10^{-2}$} & 3.45$\times 10^{-4}$ & 1.07$%
\times 10^{-3}$ & 4.26$\times 10^{-4}$ \\ \hline\hline
\end{tabular}%
\end{table}

For the errors measured in $L^{\infty }$-norm \eqref{eq:error-measure-l2} we
had similar observations (not presented here) as in the case of $L^{2}$-norm.

In summary, this example has illustrated that SGC algorithms can produce
accurate results in finding moments of solutions of nonlinear SPDE when the
integration time is relatively small. Comparing Tables \ref%
{tbl:sc-sburgers-CN-large-para} and \ref{tbl:sc-sburgers-CN}, we observe
better accuracy for the first two moments when the magnitude of noise is
smaller. In some situations higher sparse grid levels $L$ improve accuracy
but dependence of errors on $L$ is not monotone. No convergence in time step 
$h$ and in level $L$ was observed which is consistent with our theoretical
prediction in Section~2.

\begin{exm}[Stochastic advection-diffusion equation]
\upshape%
Consider the stochastic advection-diffusion equation in the Ito sense: 
\begin{gather}
du=\left( \frac{\epsilon ^{2}+\sigma ^{2}}{2}\frac{\partial ^{2}u}{\partial
x^{2}}+\beta \sin (x)\frac{\partial u}{\partial x}\right) dt+\sigma \frac{%
\partial u}{\partial x}\,d{w}(s),\quad (t,x)\in (0,T]\times (0,2\pi ),
\label{eq:perturbed--sadv-diff} \\
u(0,x)=\phi (x),\quad x\in (0,2\pi ),  \notag
\end{gather}%
where $w(s)$ is a standard scalar Wiener process and $\epsilon \geq 0$, $%
\beta $, and $\sigma $ are constants. In the tests we took $\phi (x)=\cos
(x) $, $\beta =0.1$, $\sigma =0.5,$ and $\epsilon =0.2$.
\end{exm}

We apply Algorithm \ref{algo:sadv-diff-s4-scm-mom} to %
\eqref{eq:perturbed--sadv-diff} to compute the first two moments at a
relatively large time $T=5$. The Fourier basis was taken as CONS. Since %
\eqref{eq:perturbed--sadv-diff} has a single noise only, we used
one-dimensional Gauss--Hermite quadratures of order $n.$ The implicitness
due to the use of the trapezoidal rule was resolved by the fixed-point
iteration with stopping criterion $h^{2}/100$.

As we have no exact solution of \eqref{eq:perturbed--sadv-diff}, we chose to
find the reference solution by Algorithm~4.2 from \cite{ZhangRTK12} (a
recursive Wiener chaos method accompanied by the trapezoidal rule in time
and Fourier collocation method in physical space) with the parameters: the
number of Fourier collocation points $M=30$, the length of time subintervals
for the recursion procedure $h=10^{-4}$, the highest order of Hermite
polynomials $P=4$, the number of modes approximating the Wiener process $n=4$%
, and the time step in the trapezoidal rule $h=10^{-5}$. It gives the second
moment in the $L^{2}$-norm $\left\Vert \mathbb{E}u_{\mathrm{ref}%
}^{2}(1,\cdot )\right\Vert \doteq 1.065195$. The errors are computed as
follows 
\begin{equation}
\varrho _{2}^{2}(T)=\left\vert \left\Vert \mathbb{E}u_{\mathrm{ref}%
}^{2}(T,\cdot )\right\Vert -\left\Vert \mathbb{E}u_{\mathrm{numer}%
}^{2}(T,\cdot )\right\Vert \right\vert ,\quad \varrho _{2}^{r,2}(T)=\frac{%
\varrho _{2}^{2}(T)}{\left\Vert \mathbb{E}u_{\mathrm{ref}}^{2}(T,\cdot
)\right\Vert },  \label{eq:error-measure-l2-v1}
\end{equation}%
where the norm is defined as in \eqref{eq:error-measure-l2}.

\begin{table}[tbph]
\caption{Errors in computing the second moment of the solution to the
stochastic advection-diffusion equation \eqref{eq:perturbed--sadv-diff} with 
$\protect\sigma =0.5$, $\protect\beta =0.1$, $\protect\epsilon =0.2$ at $T=5$
by Algorithm~\protect\ref{algo:sadv-diff-s4-scm-mom} with $l_{\ast }=20$ and
the one-dimensional Gauss--Hermite quadrature of order $n=2$ (left) and $n=3$
(right).}
\label{tbl:recursive-scm-advdiff-onenoise-delta-1st-order}\centering
\scalebox{0.90}{
\begin{tabular}{cccc|ccc}
\hline
$h$ & $\varrho _{2}^{r,2}(5)$ & order & CPU time (sec.) & $\varrho
_{2}^{r,2}(5)$ & order & CPU time (sec.) \\ \hline\hline
5$\times 10^{-2}$ & 1.01$\times 10^{-3}$ & -- & 7.41 & 1.06$\times 10^{-3}$
& -- & 1.10$\times 10$ \\ \hline
2$\times 10^{-2}$ & 4.07$\times 10^{-4}$ & 1.0 & 1.65$\times 10$ & 4.25$\times 10^{-4}$ & 1.0 & 2.43$\times 10$ \\ \hline
1$\times 10^{-2}$ & 2.04$\times 10^{-4}$ & 1.0 & 3.43$\times 10$ & 2.12$\times 10^{-4}$ & 1.0 & 5.10$\times 10$ \\ \hline
5$\times 10^{-3}$ & 1.02$\times 10^{-4}$ & 1.0 & 6.81$\times 10$ & 1.06$\times 10^{-4}$ & 1.0 & 1.00$\times 10^{2}$ \\ \hline
2$\times 10^{-3}$ & 4.08$\times 10^{-5}$ & 1.0 & 1.70$\times 10^{2}$ & 4.25$\times 10^{-5}$ & 1.0 & 2.56$\times 10^{2}$ \\ \hline
1$\times 10^{-3}$ & 2.04$\times 10^{-5}$ & 1.0 & 3.37$\times 10^{2}$ & 2.12$\times 10^{-5}$ & 1.0 & 5.12$\times 10^{2}$ \\ \hline\hline
\end{tabular}
}
\end{table}

The results of the numerical experiment are given in Table~\ref%
{tbl:recursive-scm-advdiff-onenoise-delta-1st-order}. We observe first-order
convergence in $h$ for the second moments. We notice that increasing the
quadrature order $n$ from $2$ to $3$ does not improve accuracy which is
expected. Indeed, the used trapezoidal rule is of weak order one in $h$ in
the case of multiplicative noise and more accurate quadrature rule cannot
improve the order of convergence. 
{\color{black}This observation confirms in some sense that the total error should be expected to be 
$O(h)$+$O(h^{L-1})$, as discussed in Section \ref{sec:recursive-sg-advdiff}.}
We note in passing that in the additive
noise case we expect to see the second order convergence in $h$ when $n=3$
due to the properties of the trapezoidal rule.

In conclusion, we showed that recursive Algorithm~\ref%
{algo:sadv-diff-s4-scm-mom} can work effectively for accurate computing of
second moments of solutions to linear stochastic advection-diffusion
equations at relatively large time. We observed convergence of order one in $%
h$.

\section*{Acknowledgments}

MVT was partially supported by the Leverhulme Trust Fellowship SAF-2012-006
and is also grateful to ICERM (Brown University, Providence) for its
hospitality. The rest of the authors were partially supported by a OSD/MURI
grant FA9550-09-1-0613, by NSF/DMS grant DMS-1216437 and also by the
Collaboratory on Mathematics for Mesoscopic Modeling of Materials (CM4)
which is sponsored by DOE. BR was also   partially supported by ARO grant W911NF-13-1-0012 
and NSF/DMS grant DMS-1148284.

 
  \def\polhk#1{\setbox0=\hbox{#1}{\ooalign{\hidewidth
  \lower1.5ex\hbox{`}\hidewidth\crcr\unhbox0}}} \def\cprime{$'$}



\begin{thebibliography}{10}

\bibitem{BabNT07}
{\sc I.~Babuska, F.~Nobile, and R.~Tempone}, {\em A stochastic collocation
  method for elliptic partial differential equations with random input data},
  {SIAM} J. Numer. Anal., 45 (2007), pp.~1005--1034 (electronic).

\bibitem{BabTZ04}
{\sc I.~Babuska, R.~Tempone, and G.~E. Zouraris}, {\em {G}alerkin finite
  element approximations of stochastic elliptic partial differential
  equations}, {SIAM} J. Numer. Anal., 42 (2004), pp.~800--825.

\bibitem{BieSch09}
{\sc M.~Bieri and C.~Schwab}, {\em Sparse high order {FEM} for elliptic
  s{PDE}s}, Comput. Methods Appl. Mech. Engrg., 198 (2009), pp.~1149--1170.

\bibitem{BudKal96}
{\sc A.~Budhiraja and G.~Kallianpur}, {\em Approximations to the solution of
  the {Z}akai equation using multiple {W}iener and {S}tratonovich integral
  expansions}, Stochastics Stochastics Rep., 56 (1996), pp.~271--315.

\bibitem{BudKal97}
\leavevmode\vrule height 2pt depth -1.6pt width 23pt, {\em The
  {F}eynman-{S}tratonovich semigroup and {S}tratonovich integral expansions in
  nonlinear filtering}, Appl. Math. Optim., 35 (1997), pp.~91--116.

\bibitem{CasLit11}
{\sc T.~Cass and C.~Litterer}, {\em On the error estimate for cubature on
  {W}iener space}, ArXiv e-prints,  (2011).

\bibitem{ComGR07}
{\sc F.~Comte, V.~Genon-Catalot, and Y.~Rozenholc}, {\em Penalized
  nonparametric mean square estimation of the coefficients of diffusion
  processes}, Bernoulli, 13 (2007), pp.~514--543.

\bibitem{DaPDT94}
{\sc G.~Da~Prato, A.~Debussche, and R.~Temam}, {\em Stochastic {B}urgers'
  equation}, NoDEA Nonlinear Differential Equations Appl., 1 (1994),
  pp.~389--402.

\bibitem{DavRab-B84}
{\sc P.~J. Davis and P.~Rabinowitz}, {\em Methods of numerical integration},
  Computer Science and Applied Mathematics, Academic Press Inc., Orlando, FL,
  second~ed., 1984.

\bibitem{Deb11}
{\sc A.~Debussche}, {\em Weak approximation of stochastic partial differential
  equations: the nonlinear case}, Math. Comp., 80 (2011), pp.~89--117.

\bibitem{Ger-Phd07}
{\sc T.~Gerstner}, {\em Sparse Grid Quadrature Methods for Computational
  Finance}, PhD thesis, University of Bonn, Habilitation, 2007.

\bibitem{GerGri98}
{\sc T.~Gerstner and M.~Griebel}, {\em Numerical integration using sparse
  grids}, Numer. Algorithms, 18 (1998), pp.~209--232.

\bibitem{GhaSpa-B91}
{\sc R.~G. Ghanem and P.~D. Spanos}, {\em Stochastic finite elements: a
  spectral approach}, Springer-Verlag, New York, 1991.

\bibitem{Gil08}
{\sc M.~B. Giles}, {\em Multilevel {M}onte {C}arlo path simulation}, Oper.
  Res., 56 (2008), pp.~607--617.

\bibitem{Gil13}
\leavevmode\vrule height 2pt depth -1.6pt width 23pt, {\em Multilevel {M}onte
  {C}arlo methods}, in Monte {C}arlo and quasi-{M}onte {C}arlo methods 2012,
  Springer, 2013.

\bibitem{GotOrs-B77}
{\sc D.~Gottlieb and S.~A. Orszag}, {\em Numerical analysis of spectral
  methods: theory and applications}, SIAM, Philadelphia, Pa., 1977.

\bibitem{GreKlo96}
{\sc W.~Grecksch and P.~E. Kloeden}, {\em Time-discretised {G}alerkin
  approximations of parabolic stochastic {PDE}s}, Bull. Austral. Math. Soc., 54
  (1996), pp.~79--85.

\bibitem{GriHol10}
{\sc M.~Griebel and M.~Holtz}, {\em Dimension-wise integration of
  high-dimensional functions with applications to finance}, J. Complexity, 26
  (2010), pp.~455--489.

\bibitem{HouLRZ06}
{\sc T.~Y. Hou, W.~Luo, B.~Rozovskii, and H.-M. Zhou}, {\em {W}iener chaos
  expansions and numerical solutions of randomly forced equations of fluid
  mechanics}, J. Comput. Phys., 216 (2006), pp.~687--706.

\bibitem{KloSho01}
{\sc P.~E. Kloeden and S.~Shott}, {\em Linear-implicit strong schemes for
  {I}t\^o-{G}alerkin approximations of stochastic {PDE}s}, J. Appl. Math.
  Stochastic Anal., 14 (2001), pp.~47--53.

\bibitem{KuoSS12a}
{\sc F.~Y. Kuo, C.~Schwab, and I.~H. Sloan}, {\em Multi-level quasi-{M}onte
  carlo finite element methods for a class of elliptic partial differential
  equations with random coefficients}, ArXiv e-prints,  (2012).

\bibitem{LitLyo12}
{\sc C.~Litterer and T.~Lyons}, {\em High order recombination and an
  application to cubature on {W}iener space}, Ann. Appl. Probab., 22 (2012),
  pp.~1301--1327.

\bibitem{LotMR97}
{\sc S.~Lototsky, R.~Mikulevicius, and B.~L. Rozovskii}, {\em Nonlinear
  filtering revisited: a spectral approach}, SIAM J. Control Optim., 35 (1997),
  pp.~435--461.

\bibitem{LotRoz06}
{\sc S.~Lototsky and B.~Rozovskii}, {\em Stochastic differential equations: a
  {W}iener chaos approach}, in From stochastic calculus to mathematical
  finance, Springer, Berlin, 2006, pp.~433--506.

\bibitem{LotRoz06a}
{\sc S.~V. Lototsky and B.~L. Rozovskii}, {\em Wiener chaos solutions of linear
  stochastic evolution equations}, Ann. Probab., 34 (2006), pp.~638--662.

\bibitem{LyoVic04}
{\sc T.~Lyons and N.~Victoir}, {\em Cubature on {W}iener space}, Proc. R. Soc.
  Lond. Ser. A Math. Phys. Eng. Sci., 460 (2004), pp.~169--198.

\bibitem{MasMon94}
{\sc G.~Mastroianni and G.~Monegato}, {\em Error estimates for
  {G}auss-{L}aguerre and {G}auss-{H}ermite quadrature formulas}, in
  Approximation and computation, Birkh\"auser Boston, 1994, pp.~421--434.

\bibitem{MikRoz98}
{\sc R.~Mikulevicius and B.~Rozovskii}, {\em Linear parabolic stochastic {PDE}s
  and {W}iener chaos}, SIAM J. Math. Anal., 29 (1998), pp.~452--480
  (electronic).

\bibitem{MilRT02}
{\sc G.~N. Milstein, Y.~M. Repin, and M.~V. Tretyakov}, {\em Numerical methods
  for stochastic systems preserving symplectic structure}, SIAM J. Numer.
  Anal., 40 (2002), pp.~1583--1604.

\bibitem{MilTre97}
{\sc G.~N. Milstein and M.~V. Tretyakov}, {\em Numerical methods in the weak
  sense for stochastic differential equations with small noise}, SIAM J. Numer.
  Anal., 34 (1997), pp.~2142--2167.

\bibitem{MilTre-B04}
\leavevmode\vrule height 2pt depth -1.6pt width 23pt, {\em Stochastic numerics
  for mathematical physics}, Springer-Verlag, Berlin, 2004.

\bibitem{MilTre09}
\leavevmode\vrule height 2pt depth -1.6pt width 23pt, {\em Solving parabolic
  stochastic partial differential equations via averaging over
  characteristics}, Math. Comp., 78 (2009), pp.~2075--2106.

\bibitem{MotNT12}
{\sc M.~Motamed, F.~Nobile, and R.~Tempone}, {\em A stochastic collocation
  method for the second order wave equation with a discontinuous random speed},
  Numer. Math.,  (2012), pp.~1--44.

\bibitem{MulRY12}
{\sc T.~Muller-Gronbach, K.~Ritter, and L.~Yaroslavtseva}, {\em Derandomization
  of the {E}uler scheme for scalar stochastic differential equations}, J.
  Complexity, 28 (2012), pp.~139--153.

\bibitem{Nie-B92}
{\sc H.~Niederreiter}, {\em Random number generation and quasi-{M}onte {C}arlo
  methods}, SIAM, Philadelphia, PA, 1992.

\bibitem{NobTem09}
{\sc F.~Nobile and R.~Tempone}, {\em Analysis and implementation issues for the
  numerical approximation of parabolic equations with random coefficients},
  Internat. J. Numer. Methods Engrg., 80 (2009), pp.~979--1006.

\bibitem{NobTW08}
{\sc F.~Nobile, R.~Tempone, and C.~G. Webster}, {\em An anisotropic sparse grid
  stochastic collocation method for partial differential equations with random
  input data}, SIAM J. Numer. Anal., 46 (2008), pp.~2411--2442.

\bibitem{NovRit99}
{\sc E.~Novak and K.~Ritter}, {\em Simple cubature formulas with high
  polynomial exactness}, Constr. Approx., 15 (1999), pp.~499--522.

\bibitem{PagPha05}
{\sc G.~Pag{\`e}s and H.~Pham}, {\em Optimal quantization methods for nonlinear
  filtering with discrete-time observations}, Bernoulli, 11 (2005),
  pp.~893--932.

\bibitem{PagPri03}
{\sc G.~Pag{\`e}s and J.~Printems}, {\em Optimal quadratic quantization for
  numerics: the {G}aussian case}, Monte Carlo Methods Appl., 9 (2003),
  pp.~135--165.

\bibitem{Pet03}
{\sc K.~Petras}, {\em {SmolPack}: a software for smolyak quadrature with
  clenshaw-curtis basis-sequence}.
\newblock
  \url{http://people.sc.fsu.edu/~jburkardt/c_src/smolpack/smolpack.html}, 2003.

\bibitem{Roz-B90}
{\sc B.~L. Rozovski{\u\i}}, {\em Stochastic evolution systems}, Kluwer, 1990.

\bibitem{SloJoe-B94}
{\sc I.~H. Sloan and S.~Joe}, {\em Lattice methods for multiple integration},
  Oxford University Press, New York, 1994.

\bibitem{Smolyak63}
{\sc S.~A. Smolyak}, {\em Quadrature and interpolation formulas for tensor
  products of certain classes of functions}, Soviet Math. Dokl., 4 (1963),
  pp.~240--243.

\bibitem{TatMcR94}
{\sc M.~Tatang and G.~McRae}, {\em Direct treatment of uncertainty in models of
  reaction and transport}, tech. rep., Department of Chemical Engineering, MIT,
  1994.

\bibitem{WasWoz95}
{\sc G.~W. Wasilkowski and H.~Wo{\'z}niakowski}, {\em Explicit cost bounds of
  algorithms for multivariate tensor product problems}, J. Complexity, 11
  (1995), pp.~1--56.

\bibitem{XiuHes05}
{\sc D.~Xiu and J.~Hesthaven}, {\em High-order collocation methods for
  differential equations with random inputs}, {SIAM} J. Sci. Comput., 27
  (2005), pp.~1118--1139.

\bibitem{XiuKar02a}
{\sc D.~Xiu and G.~Karniadakis}, {\em Modeling uncertainty in steady state
  diffusion problems via generalized polynomial chaos}, Comput. Methods Appl.
  Mech. Engrg., 191 (2002), pp.~4927--4948.

\bibitem{ZhangGun12}
{\sc G.~Zhang and M.~Gunzburger}, {\em Error analysis of a stochastic
  collocation method for parabolic partial differential equations with random
  input data}, {SIAM} J. Numer. Anal., 50 (2012), pp.~1922--1940.

\bibitem{ZhangRTK12}
{\sc Z.~Zhang, B.~Rozovskii, M.~V. Tretyakov, and G.~E. Karniadakis}, {\em A
  multistage {W}iener chaos expansion method for stochastic
  advection-diffusion-reaction equations}, SIAM J. Sci. Comput., 34 (2012),
  pp.~A914--A936.

\end{thebibliography}
\end{document}